\documentclass[10pt,a4paper]{article}
\usepackage{amsbsy}
\usepackage{amssymb}
\usepackage{graphicx}
\usepackage{amsmath}
\usepackage{epsfig}
\usepackage{amsfonts}
\usepackage{amsthm}

\newtheorem{theo}{Theorem}[section]
\newtheorem{lem}{Lemma}

\theoremstyle{definition}
\newtheorem{definition}{\bf Definition}
\theoremstyle{remark}

\newtheorem{cor}{\bf Corollary}
\newtheorem{prop}{\bf Proposition}

\begin{document}
\title{SPACE-TIME  FRACTIONAL STOCHASTIC EQUATIONS  ON  REGULAR BOUNDED OPEN DOMAINS}
\date{}
\maketitle
\vspace*{1cm}

\noindent {\bfseries V.V. Anh$^{1},$ N.N. Leonenko$^{2}$ and M.D. Ruiz-Medina$^{3}$}

\noindent{$^{1}$\ School of Mathematical Sciences, Queensland University of Technology, GPO Box 2434,\\
Brisbane QLD 4001, Australia; e-mail: v.anh@qut.edu.au\\
$^{2}$ School of Mathematics, Cardiff University, Senghennydd Road,\\ Cardiff CF2 4YH, UK; e-mail: leonenkon@cardiff.ac.uk\\
$^{3}$ Faculty of Sciences, Campus Fuente Nueva s/n,\\ Granada University, 18071 Granada, Spain; e-mail: mruiz@ugr.es}

%%%%%%%%%%%%%%% begin make title %%%%%%%%%%%%%%

 %%% TITLE: texts in [.] is abbreviated (1st line) title for running heads
 %%% Author(s): put in brackets [.] the short author's name

%%%%%%%%%%%%%%%%%%%%%%%%%%%%%%%%%%%%%%%%%%%%%%%%%%%%%%%%%%%%%%%%%%%%%%%
                    % THE BEGINNING %

 \bigskip \medskip
\noindent \emph{This paper is now published (in revised form) in \textbf{Fractional Calculus and Applied Analysis}, \textbf{Vol. 19, pp. 1161--1199,}  
DOI: 10.1515/fca-2016-0061, and  is available online at\\  http://www.degruyter.com/view/j/fca. }
%%%% Abstract %%%%%%%%%%%%%%%%%%%%%%%%%
 \begin{abstract}

Fractional (in time and in space) evolution  equations defined  on
Dirichlet regular bounded open domains, driven by fractional
integrated in time  Gaussian spatiotemporal white noise,   are
considered here. Sufficient conditions for the definition of a
weak-sense Gaussian solution, in the mean-square sense,  are
derived. The temporal, spatial and spatiotemporal H\"older
continuity, in the mean-square sense,  of the formulated  solution is
obtained, under suitable conditions,  from the asymptotic properties
of the Mittag-Leffler function, and the asymptotic order of the
eigenvalues of a fractional polynomial of the Dirichlet
  negative Laplacian operator on such bounded open domains.

 \medskip

{\it MSC 2010\/}: Primary 60G60, 60G15, 60G22; Secondary 60G20,
60G17, 60G12.

 \smallskip

{\it Key Words and Phrases}: Caputo-Djrbashian fractional-in-time
derivative,
 Dirichlet regular bounded
open domains, eigenfunction expansion, fractional pseudodifferential
elliptic operators, Gaussian spatiotemporal white noise measure,
Mittag-Leffler function, Riemannan-Liouville fractional integral and
derivative, stochastic boundary value problems

 \end{abstract}

 \maketitle

%%%%%%% end make title %%%%%%%%%%%%%%%%%%%%%%%%%%%%%%%%%%
 \vspace*{-16pt}

%%%%%%%% begin papers' body %%%%%%%%%%%%%%%%%%%%%%%%%%%%%

%%%%%%%%%%%%%%%%%%%%%%%%%%% Section 1 %%%%%%%%%%%%%
%\section{Introduction}\label{Sec:1}

 \section{Introduction}\label{sec:1}

\setcounter{section}{1}

Space-time fractional diffusion equations are introduced when
integer-order derivatives in space and in time are replaced by their
fractional counterpart (see, for example, \cite{Pagnini}). In particular, they can model anomalous
diffusion processes in physics (Meerschaert \emph{et al.}
\cite{Meerschaert02}). Fractional diffusion equations are very
popular in several fields of application (see Gorenflo and Mainardi
\cite{Gorenflo03};
 Metzler and Klafter \cite{Metzler04}, among others).

The development of space-time fractional diffusion equations is supported by
the theory of operators for generalized fractional calculus (see, for example,  Kiryakova \cite{Kiryakova}, \cite{Kiryakova2}; Machado, Kiryakova and Mainardi \cite{Machado}, \cite{Machado0}, and the references therein).

Since the pioneer papers by Bochner \cite{Bochner49} and Feller
\cite{Feller52} who proved the connection between the stable
distribution and fractional calculus, the theory of $\alpha $-stable
distributions and processes has been extensively developed.
Specifically, Bochner \cite{Bochner49} formulated the Cauchy
problem, whose solution is the symmetric $\alpha $-stable
distribution. Feller \cite{Feller52} extended these results to a
more general situation by replacing the fractional Laplacian
$-\left( -\Delta \right) ^{\alpha /2}$ by a pseudodifferential
operator with symbol
$$-\left\vert \lambda \right\vert ^{\alpha }\exp \left(
i\,\text{sign}\left( \lambda \right) \theta \pi
/2\right) ,\quad \lambda \in \mathbb{R},\quad \alpha \in (0,2),$$ \noindent where $\alpha $ is the index of stability, and $%
\theta $ is the index of skewness (asymmetry). The corresponding
solutions generate all stable distributions. The Cauchy problem characterizing the main properties
of the Riesz-Bessel distribution, generated by the pseudodifferential
operator $(-\Delta )^{\alpha /2}(I-\Delta )^{\gamma /2},$ was analyzed in
Anh, Leonenko and Sikorskii \cite{AnhLeonenki14}. In spite of a large number
of fractional operators (see Samko \emph{et al.} \cite{Samko93}), there were
few known specific examples producing more general distributions.

The traditional model for particle spreading at the macroscopic level is the
well-known heat equation
\begin{equation*}
\partial _{t}u=\Delta u,
\end{equation*}%
\noindent with $\Delta $ denoting the Laplacian operator and $\partial _{t}$
the partial derivative in time. The relative particle concentration can be
predicted in terms of the Gaussian probability density providing a point
source solution of the heat equation. The paths of individual particles are
described in terms of the realizations of Browninan motion. Dirichlet
boundary value problems for the heat equation, as well as for more general
equations, given in terms of elliptic diffusion operators, are detailed in
Bass \cite{Bass98} and Davies \cite{Davies89}, among others. The phenomena
of particle sticking and trapping can be described by using the fractional
derivative $\partial _{t}^{\beta }$ for $0<\beta <1$ instead of the partial
derivative in time $\partial _{t}.$ On the other hand, long particle jumps
can be represented by the fractional power $(-\Delta )^{\alpha /2},$ for $%
0<\alpha <2,$ in place of the negative Laplacian $(-\Delta )$. The
space-time fractional diffusion equation is then defined in terms of both
fractional derivatives in time and in space:
\begin{equation}\partial_{t}^{\beta }u=(-\Delta )^{\alpha
/2}u,\label{eqpde}\end{equation} \noindent whose solution displays
self-similarity and heavy tails. The particle concentration profile
provided by the corresponding probability density solution has
sharper peak and heavy tails. A non-Markovian setting can then be
introduced through time change by an
inverse stable subordinator (see
also Barkai \emph{et al.} \cite{BarkaiMetzler00}; Benson \emph{et
al.} \cite{BensonWheatcraft00}; Gorenflo  and Mainardi
\cite{Gorenflo99}; \cite{Gorenflo03};  Meerschaert \emph{et al.}
 \cite{Meerschaert02}; Schneider and Wyss \cite{Schneider89}, among others). An extension to the case of Riesz-Bessel
subordinators was addressed in Anh and McVinish \cite{AnhMcVinish04}, using
the pseudodifferential operator $(-\Delta )^{\alpha /2}(I-\Delta )^{\gamma
/2}.$ A  different stochastic framework is
analyzed in the papers by Anh and Leonenko \cite{AnhLeonenki01},
\cite{AnhLeonenki03}, where the spectral representation of the
mean-square solution of the following stochastic space-time fractional
 equation with random initial conditions
\begin{equation*}\partial_{t}^{\beta }u=(-\Delta
)^{\alpha /2}\left( I-\Delta \right) ^{\gamma /2}u, \quad
u_{0}(\mathbf{x})=\eta (\mathbf{x}),\quad \mathbf{x}\in
\mathbb{R}^{n},\label{eqpde}\end{equation*} \noindent
 is derived. Here,  $\eta $ is a measurable random field defined on a
complete probability space $(\Omega, \mathcal{A},P).$ Gaussian and
non-Gaussian limiting distributions of the renormalized solution are
obtained as well. Also, in the context of stochastic evolution
equations on an unbounded domain, a functional approach was adopted
by Kelbert, Leonenko and Ruiz-Medina \cite{Kelbert05}, where the
spectral properties of the mean-square solution  of fractional in
time and in space evolution equations driven by random white noise
are derived. These results are extended to the more general
framework of stochastic partial differential equations driven by
fractional Brownian motion in Leonenko, Ruiz-Medina, Taqqu
\cite{Leonenko11}, where the spectral correlation structure
 of the mean-square solution to fractional  space-time evolution equations driven by fractional Brownian
are analyzed. In the particular case of the domain being an
$n$-dimensional rectangle, starting from the pioneering papers on the
stochastic heat equation by Walsh \cite{Walsh}, \cite{Walsh2}, Angulo,
Ruiz-Medina, Anh and Grecksch \cite{Angulo2} considered the following
fractional (in space) version of the heat equation:
\begin{equation}
\frac{\partial }{\partial t}c\left( t,x\right) +\left( -\Delta
\right) ^{\alpha /2}\left( I-\Delta \right) ^{\gamma /2}c\left(
t,x\right) =\varepsilon \left( t,x\right) ,  \quad \label{1.1}
\end{equation}%

\noindent for  $t\in \mathbb{R}_{+},\ x\in D=(0,L_{1})\times\dots
\times  (0,L_{n}),$ and $n<\alpha +\gamma <n+2,$ where $\varepsilon
\left( t,x\right) $ is a zero-mean Gaussian space-time white noise.
The space-time mean-quadratic and sample-path local variation
properties of the solution are derived in this paper. The unbounded
domain case is addressed as well.  Angulo,  Anh,  McVinish
and Ruiz-Medina \cite{Angulo1} derive the extension of these results to the
case of fractional derivatives in time and space with the same
spatial pseudodifferential operator $\left( -\Delta \right) ^{\alpha
/2}\left( I-\Delta \right) ^{\gamma /2}$ and $D=(0,L_{1})\times\dots
\times  (0,L_{n})$ (see also the references therein).

In the context of fractional diffusion on bounded domains, we refer
to the papers by Defterli, D'Elia,  Du,  Gunzburger, Lehoucq and
Meerschaert \cite{Defterli}; Chen,  Meerschaert and   Nane
\cite{ChenMeerschaert12}, and Meerschaert,  Nane and Vellaisamy
\cite{Meerschaert03}, where strong solutions, and their
probabilistic representation are obtained. On the other hand, Mijena
and Nane \cite{MijenaNane} consider fractional heat equation on
unbounded domains, with a non-linear random external force,
involving space-time white noise. Sufficient conditions for the
existence and uniqueness of mild solutions, as well as for their
continuity are derived. We consider here a different framework.
Specifically, we study the weak-sense solution of the following
fractional in space and in time stochastic partial differential
equation, with Dirichlet boundary conditions, and null initial
condition:
\begin{eqnarray} & &\hspace*{-1.5cm}\frac{\partial ^{\beta }}{\partial
t^{\beta }}c\left( t,\mathbf{x}\right) +\left( -\Delta_{D} \right)
^{\alpha /2}\left( I-\Delta_{D} \right) ^{\gamma
/2}c\left( t,\mathbf{x}\right) =I^{1-\beta }_{t}\varepsilon \left( t,\mathbf{x}\right),\ \mathbf{x}\in D \label{1.2}\\
& & c(t,\mathbf{x}) = 0,\quad \mathbf{x}\in \partial D,\quad \forall
t,\quad c(0,\mathbf{x})=0,\quad \forall \mathbf{x}\in D\subset
\mathbb{R}^{n}, \label{1.2b}
\end{eqnarray}
\noindent \noindent for $\beta \in (0,1),$ $\alpha +\gamma >n,$
where equality is understood in the mean-square sense.  Here, the
driven process
\begin{equation}
I^{1-\beta }_{t}\varepsilon =\frac{1}{\Gamma (1-\beta
)}\int_{0}^{t}(t-u)^{-\beta }\varepsilon(u)du\label{fiFLeq}
\end{equation}
\noindent is constructed from space-time Gaussian white noise
$\varepsilon ,$ defined on a basic probability space $(\Omega
,\mathcal{A},P),$ and satisfying $$E[\varepsilon \left(
t,\mathbf{x}\right)\varepsilon \left( s,\mathbf{y}\right)]=\delta
(t-s)\delta (\mathbf{x}-\mathbf{y}),$$\noindent  for all $t,s\in
\mathbb{R}_{+},$ and  $\mathbf{x},\mathbf{y}\in D,$ with $\delta $
being the Dirac Delta distribution. Specifically,  the driven
process is constructed from  fractional integration of order $\beta
-1,$ in time,    of the space-time Gaussian white noise $\varepsilon
,$ where integration is understood in the mean-square sense (see,
for example, Samko \emph{et al.} \cite{Samko93}). It is well-known
that the inverse of the  fractional integral of order $\beta -1$ is
the  fractional derivative of order $1-\beta .$
This fractional derivative, considered in
this paper, is the regularized fractional derivative in time or
fractional-in-time derivative in the Caputo-Djrbashian sense: For $\beta \in
(0,1],$
\begin{equation}
\frac{\partial ^{{}\beta }u}{\partial t^{{}\beta }}=\left\{
\begin{array}{ll}
\frac{\partial u}{\partial t}\left( t,\mathbf{x}\right) , & \text{if }\beta
=1 \\
\frac{1}{\Gamma \left( 1-\beta \right) }\frac{\partial }{\partial t}%
\int_{0}^{t}\left( t-\tau \right) ^{-\beta }u\left( \tau ,\mathbf{x}\right)
d\tau -\frac{u(0,\mathbf{x})}{t^{\beta }}, & \text{if}\ \beta \in (0,1),\
t\in (0,T]%
\end{array}%
\right.   \label{1.3}
\end{equation}%
\noindent (see Meerschaert and Sikorskii \cite{Meerschaert02b}; Podlubny
\cite{Podlubny99}).

Although we refer here to the particular case where fractional
derivatives in space are defined from the operator $\left(
-\Delta_{D} \right) ^{\alpha /2}\left( I-\Delta_{D} \right) ^{\gamma
/2},$ with $(-\Delta_{D})$ representing the Dirichlet negative
Laplacian operator on regular bounded open domain $D,$ the results
derived in this paper hold, in general, for  a  fractional
polynomial of the Dirichlet negative Laplacian operator on $D,$ with
constant coefficients, as proved in Theorem \ref{thfg} in Section
\ref{FC1}. In this paper, special attention has been  paid to
operator $\left( -\Delta_{D} \right) ^{\alpha /2}\left( I-\Delta_{D}
\right) ^{\gamma /2},$ since, for suitable domains, e.g., for
bounded open domain satisfying the exterior cone condition, the
eigenvalues of such an operator provide two-sided estimates of the
eigenvalues of the corresponding restriction of the inverse of the
composition of Riesz and Bessel potentials, for certain  range of
parameter $\alpha $ (see, for example, Chen and  Song
\cite{ChenSong05}).

Our main goal is the study of the local regularity (modulus of continuity)
of the derived weak-sense Gaussian solution to equations (\ref{1.2})--(\ref{1.2b}).
Sufficient conditions are formulated  to obtain the mean-quadratic
local asymptotic order of the temporal, spatial and spatiotemporal
increment random fields, associated with the weak-sense Gaussian
solution to equations (\ref{1.2})--(\ref{1.2b}) (see Theorems
\ref{prtqv}, \ref{th2} and \ref{stinc} below). Specifically, the
results derived hold under the condition that the regular bounded
open domain $D$  is such that the eigenvectors of the Dirichlet
negative  Laplacian operator on $D$ are uniformly bounded. Some
examples of domains $D,$ where this condition is satisfied, are
provided in Section \ref{examples}. Furthermore, the mean-square
H\"older continuity in time  of the random field solution  is
obtained under some restrictions on the parameter space. While its
mean-square H\"older continuity in space  requires the H\"older
continuity of the eigenvectors of the Dirichlet negative Laplacian
operator on domain $D.$ The mean-square H\"older continuity in space
and time directly follows, under the above  conditions.
Also, under such conditions,
the sample-path local asymptotic orders are obtained immediately from
Theorem 3.3.3 of Adler \cite{Adler81} (see Theorem \ref{thspp} below). Note
that, although the time fractional differentiation in equation (\ref{1.2})
is understood in the weak-sense, the Gaussian solution $c$ is H\"{o}lder
continuous in the mean-square sense under the conditions assumed in this
paper.

This paper does not adopt the classical framework of diffusion processes
characterized by the Kolmogorov forward equation or Fokker--Planck equation
(see, for example, \cite{Hahn}). In our case, the regularized fractional
derivative in time, or fractional-in-time derivative in the
Caputo-Djrbashian sense, and the Fokker--Planck operator with constant
coefficients are applied, in the mean-square sense, to a spatiotemporal
Gaussian random field for its \emph{almost decorrelation in space and time}.
Hence, local self-similarity is observed in the correlation structure in
space and in time of the weak-sense mean-square Gaussian solution $c,$ as we
will prove in this paper. The approach adopted is then different from that
considered in Chen, Meerschaert and Nane \cite{ChenMeerschaert12}, since, in
the latter approach, the properties of the transition probability densities
are investigated, while, in this paper, new classes of spatiotemporal
Gaussian random fields, displaying local self-similarity, are introduced in
the weak sense. In particular, the exponents of their local self-similarity
are computed in time, space and space-time, in the mean-square and
sample-path sense.

Finally, we recall the interest of considering especially spatiotemporal
Gaussian random fields on Dirichlet regular bounded open domains (see
Fuglede \cite{Fuglede05}), including, as particular cases, bounded open $%
C^{\infty }$- domains, domains with $C^{1}$-boundary, with Lipschitz
continuous boundary, or with fractal boundary, among others. Special
attention , in the current literature, has been paid to the unit ball and
the unit sphere, motivated by the analysis of Cosmic Microwave Background
(CMB) radiation (see, for example, Leonenko and Sakhno \cite{Leonenko12};
Malyarenko \cite{Malyarenko12}; Marinucci and Peccati \cite{Marinucci11}).
In this setting, tensor-valued random fields on the unit sphere are
considered for the investigation of the combinations $Q\pm \ iU,$ with $Q$
and $U$ respectively representing the linear and circular polarization
Stokes parameters.

The outline of the paper is as follows. Preliminary elements and
results are presented in Section \ref{Sec2}.   The derivation of a
weak-sense mean-square Gaussian solution to equations
(\ref{1.2})-(\ref{1.2b}) is established in Section \ref{wmssol}.
  The
mean-quadratic local variation exponents in time of the derived
solution are obtained in Section \ref{Sec3}. The mean-quadratic
local variation exponents in space are given in Section \ref{Sec4}.
Section \ref{Sec5} then provides the asymptotic local mean quadratic
orders in space and time. The modulus of continuity of the sample
paths of the weak-sense mean-square solution to equations
(\ref{1.2})--(\ref{1.2b}) is also derived in this section . Some
examples are provided in Section \ref{examples} for illustration
purposes. The extended formulation of the results derived for
fractional polynomials of the Dirichlet negative Laplacian operator
are presented in Section \ref{FC1}. Final comments and some open
research lines are discussed in Section \ref{FC}.

\section{Preliminaries}
\label{Sec2}

\setcounter{section}{2}

Some preliminary definitions and results needed in the development
of this paper are now introduced. Specifically, some basic results
on spectral calculus for self-adjoint operators on a Hilbert space
are given in Section \ref{sec:21}. The  Mittag-Leffler function is
conisdered in Section \ref{MLF}. Basic elements on fractional
Sobolev spaces on a regular  bounded open domain are presented  in
Section \ref{fssrbod}.

 \subsection{Spectral theory of self-adjoint  operators on a separable Hilbert space}
\label{sec:21}

Let us first consider some results on spectral calculus for
self-adjoint operators on a Hilbert space.
\begin{theo}
\label{theorem:1} (Dautray and Lions, 1990, pp. 119-120
\cite{Dautray90}) Let $H$ be a separable Hilbert space, then an
injection mapping $\widehat{\sigma}$
 exists from the set of spectral families in $H$ into the set of self-adjoint operators on $H$. The following assertions hold:

 Let $\mathbb{A}$ be the self-adjoint operator associated with the spectral family $\left\lbrace E_\lambda \right\rbrace_{\lambda \in \Lambda},$ where
 $\Lambda $ denotes the  spectrum of $\mathbb{A}.$ The domain of $\mathbb{A}^k$ is defined by

\begin{equation}
D\left(\mathbb{A}^k\right) = \left\lbrace x \in H:~\displaystyle
\int_{\Lambda } \lambda^{2k}d\left(E_\lambda x, x\right) < \infty
\right\rbrace,~k \geq 1. \label{1}
\end{equation}
\noindent For all $x\in D\left(\mathbb{A}^k\right),$ and for all $y
\in H,$

\begin{eqnarray}
\langle \mathbb{A}^k x,y\rangle_H &=& \displaystyle \int_{\Lambda }
\lambda^k d\left(E_\lambda x, y\right), \\ \label{2} \Vert
\mathbb{A}^k x \Vert_{H}^{2} &=& \displaystyle \int_{\Lambda}
\lambda^{2k} d\left(E_\lambda x, x\right). \label{3}
\end{eqnarray}

 If $P_k \left( \lambda \right)$ is a polynomial of degree $k,$ then, for all $x \in D\left(\mathbb{A}^k \right),$ and for all $y \in H$, $P_k \left(\mathbb{A}\right)$ is given by

\begin{equation}
\langle P_k \left(\mathbb{A} \right) x,y \rangle_H = \displaystyle
\int_{\Lambda} P_k \left( \lambda \right)d\left(E_\lambda x, y
\right). \label{4}
\end{equation}
\noindent Finally, for a continuous function  $f$  on $\Lambda,$ the
following identities hold for every $x\in D\left(f\left(\mathbb{A}
\right) \right),$ and $y \in H,$

\begin{equation}
\langle f\left(\mathbb{A} \right) x, y \rangle_H = \displaystyle
\int_{\Lambda} f\left( \lambda \right) d\left(E_\lambda x,y \right).
\label{5}
\end{equation}

\end{theo}

\begin{theo}
\label{theorem:2} (Dautray and Lions, 1990, p. 140
\cite{Dautray90}) Let $\mathbb{A}$ be a self-adjoint operator in a
separable Hilbert space $H$. If we denote $\overline{f}$ the complex
conjugate function for $f$, then $D\left(\overline{f}\left(
\mathbb{A} \right)\right) = D\left(f\left(
\mathbb{A}\right)\right)$. Moreover we have
 $\langle f\left(\mathbb{A}\right)x,y\rangle_H = \langle x^{*}, \overline{f}\left(\mathbb{A}\right)y^{*} \rangle,\ $ for all $x,~y \in D\left(f\left(\mathbb{A}\right)\right)$.

 For $x \in D\left(f\left(\mathbb{A}\right)\right),$ and $y \in D\left(g\left(\mathbb{A}\right)\right),$ then

\begin{equation}
\langle f\left(\mathbb{A}\right)x,g\left(\mathbb{A}\right)y\rangle_H
= \displaystyle \int_{\Lambda} f\left(\lambda\right)
\overline{g}\left(\lambda\right) d\left(E_\lambda x,y\right).
\label{6}
\end{equation}

\noindent Furthermore,  $\left(f + g\right)\left(\mathbb{A}\right)x
= f\left(\mathbb{A}\right)x + g\left(\mathbb{A}\right)x,\ $ for all
$x \in D\left(f\left(\mathbb{A}\right)\right) \cap
D\left(g\left(\mathbb{A}\right)\right).$

\noindent Finally,  if $x \in
D\left(f\left(\mathbb{A}\right)\right)$, with $\left(g \circ
f\right)\left(\lambda \right) =
g\left(\lambda\right)f\left(\lambda\right)$, then $\left[
g\left(\mathbb{A}\right)f\left(\mathbb{A}\right)\right] x = \left( g
\circ f \right) \left(\mathbb{A}\right)x$.

\end{theo}

Theorem \ref{theorem:1} is now applied to derive the asymptotic
order of the eigenvalues of operator $ (-\Delta)_{D}^{\alpha
/2}(I-\Delta_{D} )^{\gamma /2},$ with, as before,  $(-\Delta)_{D}$
representing the Dirichlet negative Laplacian operator on regular
bounded open domain $D.$

\begin{cor}
\label{cor1} The following asymptotic order holds for the
eigenvalues of $(I-\Delta_{D} )^{\gamma /2}(-\Delta_{D})^{\alpha
/2}:$

\begin{equation}
\lim_{k\longrightarrow \infty }\frac{\lambda _{k}\left(
(-\Delta_{D})^{\alpha /2}(I-\Delta_{D} )^{\gamma /2}\right)
}{k^{\alpha +\gamma /n}}=\widetilde{c}(n,\alpha +\gamma
)|D|^{-(\gamma +\alpha )/n}, \label{eqfa2bbtt}
\end{equation}
\noindent where $\widetilde{c}(n,\alpha +\gamma )$ is a positive
constant depending on $n,$ $\alpha $ and $\gamma.$

Futhermore, for $\{\phi_{k}\}_{k\geq 1}$ being the eigenvector
system  of the Dirichlet negative Laplacian operator $(-\Delta_{D}
)$ on domain $D,$ the following equality holds:
\begin{equation}
(-\Delta_{D})^{\alpha /2}(I-\Delta_{D} )^{\gamma /2}\phi_{k}=\lambda
_{k}\left((-\Delta_{D})^{\alpha /2}(I-\Delta_{D} )^{\gamma
/2}\right)\phi_{k},\quad k\geq 1. \label{eqeivectorL}
\end{equation}

\end{cor}
\begin{proof}
It is well-known that the eigenvalues $\{\gamma_{k}(-\Delta_{D}
)\}_{k\geq 1}$ of the Dirichlet negative Laplacian operator on
domain $D\subset \mathbb{R}^{n},$ arranged in decreasing order of
their modulus magnitude satisfy (see, for example,  Chen and  Song
\cite{ChenSong05}):
\begin{equation}
\gamma_{k}(-\Delta_{D} )\sim 4\pi\frac{\left(\Gamma \left( 1+\frac{n}{2}%
\right)\right)^{2/n}}{|D|^{2/n}}k^{2/n},\quad k\longrightarrow
\infty , \label{tseLapac}
\end{equation}
\noindent where $f(k)\sim g(k)$ means that $\lim_{k\rightarrow
\infty}f(k)/g(k)=C,$ for certain positive constant $C.$ In
particular, $C=1$ in (\ref{tseLapac}).

From equation (\ref{5}) in Theorem \ref{theorem:1}, considering
$f(u)= u^{\alpha /2}(1+u)^{\gamma /2},$ we  obtain
\begin{equation} \lambda _{k}\left((-\Delta_{D})^{\alpha
/2}(I-\Delta_{D} )^{\gamma /2}\right)= \left(\gamma_{k}(-\Delta_{D}
)\right)^{\alpha /2}(1+\gamma_{k}(-\Delta_{D} ))^{\gamma /2}.
\label{fs}
\end{equation}

 Equation (\ref{eqfa2bbtt}) then follows from equations (\ref{tseLapac}) and (\ref{fs}).

Equation (\ref{eqeivectorL}) is straightforwardly obtained from
equation (\ref{5}) in Theorem \ref{theorem:1}, since in our case,
i.e., for $f(u)= u^{\alpha /2}(1+u)^{\gamma /2},$ for all $x,y\in
H=L^{2}(D),$ with  $L^{2}(D)$ denoting the space of square
integrable functions on $D,$
\begin{eqnarray}
\int_{\Lambda }f(\lambda )d\left(E_\lambda x,y
\right)&=&\sum_{k=1}^{\infty}\left(\gamma_{k}(-\Delta_{D}
)\right)^{\alpha /2}(1+\gamma_{k}(-\Delta_{D} ))^{\gamma /2} \nonumber\\
& &\times \int_{D\times
D}\phi_{k}(\mathbf{u})\phi_{k}(\mathbf{v})x(\mathbf{u})y(\mathbf{v})d\mathbf{u}d\mathbf{v}
\nonumber\\
&=&\sum_{k=1}^{\infty}\left(\gamma_{k}(-\Delta_{D} )\right)^{\alpha
/2}(1+\gamma_{k}(-\Delta_{D} ))^{\gamma /2}x_{k}y_{k},\nonumber\\
& & \label{ri}
\end{eqnarray}
\noindent with, as before,  $\{\phi_{k}\}_{k\geq 1}$ being the
eigenvector system  of the Dirichlet negative Laplacian operator on
$D.$ Specifically, our spectral family is defined in terms of the
spectral kernel
$\sum_{k=1}^{\infty}\phi_{k}(\mathbf{u})\phi_{k}(\mathbf{v}),$ and
the spectral measure is given by a point or counting measure with
atoms located at the eigenvalues.
 \hfill $\blacksquare$
\end{proof}

\subsection{Mittag-Leffler function}
\label{MLF} The weak-sense solution derived in the next section
involves the Mittag-Leffler function. The  definition  of  the
Mittag-Leffler function, and a two-sided uniform inequality are now
considered.
\begin{definition}
\label{lem0} The Mittag-Leffler function is given by
\begin{equation}
E_{\beta }(\mathbf{z})=\sum_{j=0}^{\infty
}\frac{(\mathbf{z})^{j}}{\Gamma (j\beta+1)},\quad
\mathbf{z}\in\mathbb{C},\quad 0<\beta \leq 1 \label{mlf}
\end{equation}

\noindent (see Erd\'elyi \emph{et al.} \cite{Erdely55}; Haubold,
 Mathai and  Saxena \cite{Haubold11}, for a more detailed
description of this function and its properties).
\end{definition}

\begin{lem}
\label{th4s} For every $\beta \in (0,1),$ the uniform estimate
\begin{equation*}
\frac{1}{1+\Gamma (1-\beta )x}\leq E_{\beta }(-x)\leq
\frac{1}{1+[\Gamma (1+\beta )]^{-1}x}
\end{equation*}%
\noindent holds over $\mathbb{R}_{+}$ with optimal constants (see
Simon \cite{Simon14}, Theorem 4).

\end{lem}

\subsection{Fractional Sobolev spaces on regular bounded open domains}
\label{fssrbod} The scale of fractional Sobolev spaces is introduced
within the
 spaces $\mathcal{S\,}\left( \mathbb{R}%
^{n}\right) ,$ the space of $C^{\infty }$-functions
with rapid decay at infinity, and $\mathcal{D}\left( \mathbb{R}%
^{n}\right),$ the space of $C^{\infty }$-functions with compact
support contained in $\mathbb{R}^{n}.$ The dual of these spaces are
respectively the
space of tempered distributions, $\mathcal{S}^{\prime }\left( \mathbb{R}%
^{n}\right) $, and the space of distributions, $\mathcal{D}^{\prime
}\left( \mathbb{R}^{n}\right).$

For $s\in \mathbb{R}$, we denote by $H^{s}\left(
\mathbb{R}^{n}\right) $ the space of tempered distributions $u$ such
that $\left( 1+\|\boldsymbol{\lambda
}\|^{2}\right)^{s/2}\widehat{u}\in L_{2}\left( \mathbb{R}^{n}\right)
,~\boldsymbol{\lambda }\in \mathbb{R}^{n}.$ For a regular bounded
open domain $D$ in $\mathbb{R}^{n}$, we denote
\begin{equation}
\overline{H}^{s}\left( D\right) =\left\{ u\in H^{s}\left( \mathbb{R}%
^{n}\right) :\mathrm{supp}\,\,u\subseteq
\overline{D}\right\},\label{eqsob1}
\end{equation}
\begin{equation}
H^{s}\left( D\right) =\left\{ f\in \mathcal{D}^{\prime }\left(
D\right) :\exists F\in H^{s}\left( \mathbb{R}^{n}\right) \text{ such
that }f=F_{D}\right\}, \label{eqfss}
\end{equation}%
\noindent where $F_{D}$ denotes the restriction of $F$ to $D$. With
the quotient norm
\begin{equation*}
\left\Vert f\right\Vert _{H^{s}\left( D\right) }=\underset{\left\{
F;F_{D}=f\right\} }{\inf }\,\left\Vert F\right\Vert _{H^{s}\left( \mathbb{R}%
^{n}\right) },
\end{equation*}%
$H^{s}\left( D\right) $ is a Hilbert space (see Dautray and Lions,
\cite{Dautray90}, p. 118).

\section{The mean-square Gaussian solution in the weak sense}
\label{wmssol}

\setcounter{section}{3}

The preliminaries given in the previous section are now applied in
the derivation of a zero-mean Gaussian solution to the stochastic
boundary value problem  (\ref{1.2})--(\ref{1.2b}), in the
mean-square and weak senses. The following result first establishes
the suitable range of parameter $\alpha $ and $\gamma $ for the
construction of a Green operator in the  trace  class,  with kernel,
the fundamental solution to the deterministic problem corresponding
to (\ref{1.2})--(\ref{1.2b}). Namely, the following proposition
states the  ranges of parameters $\alpha $ and $\gamma $ such that
the sequence
\begin{equation}\left\{E_{\beta }\left( -\lambda _{k}\left( (-\Delta_{D} )^{\alpha
/2}(I-\Delta_{D})^{\gamma /2}\right) t^{\beta }\right),\ k\geq
1\right\}\label{eimlf} \end{equation}\noindent  is in the space
$l^{1}$ of absolute summable sequences, for every $t>0.$
\begin{prop}
\label{pr1} For $n<\alpha +\gamma ,$ \begin{equation}
\sum_{k=1}^{\infty }E_{\beta }\left( -\lambda _{k}\left(
(-\Delta_{D} )^{\alpha /2}(I-\Delta_{D} )^{\gamma /2}\right)
t^{\beta }\right) <\infty ,  \label{tracep0}
\end{equation}

\noindent  for every $t>0.$
\end{prop}

\begin{proof}
From  equation (\ref{eqfa2bbtt}),
\begin{equation}\lim_{k\longrightarrow \infty }\frac{\lambda_{k}\left((-\Delta_{D}
)^{\alpha /2}(I-\Delta _{D})^{\gamma /2}\right)}{k^{(\alpha+\gamma
)/n}}=\widetilde{c}(n,\alpha +\gamma )|D |^{-(\alpha +\gamma )
/n}.\label{jao}
\end{equation}
Therefore, there exists $k_{0}$ such that for $k\geq k_{0},$
\begin{equation}L_{1}k^{(\alpha+\gamma )/n}\leq \lambda_{k}\left((-\Delta_{D} )^{\alpha /2}(I-\Delta_{D} )^{\gamma
/2}\right)\leq L_{2} k^{(\alpha+\gamma )/n},\label{ipp}
\end{equation} \noindent for
certain positive constants $0<L_{1}<L_{2},$ depending on $k_{0},$
and $\alpha +\gamma $ and $n.$  In particular, for $k\geq k_{0},$
\begin{equation}\frac{1}{1+[\Gamma (1+\beta
)]^{-1}\lambda_{k}\left((-\Delta_{D})^{\alpha /2}(I-\Delta_{D}
)^{\gamma /2}\right)t^{\beta }}\leq \frac{1}{1+[\Gamma (1+\beta
)]^{-1} L_{1}k^{(\alpha+\gamma )/n}t^{\beta }}. \label{eivaso3}
\end{equation}

Now, from equations (\ref{ipp}) and (\ref{eivaso3}), applying  Lemma \ref{th4s}, for each fixed $t>0,$

\begin{eqnarray}
& & \sum_{k=1}^{\infty }E_{\beta }\left(-t^{\beta
}\lambda_{k}\left((-\Delta_{D})^{\alpha /2}(I-\Delta_{D} )^{\gamma
/2}\right)\right)\nonumber\\
& &=\sum_{k=1}^{k_{0}}E_{\beta }\left(-t^{\beta
}\lambda_{k}\left((-\Delta _{D})^{\alpha /2}(I-\Delta_{D} )^{\gamma
/2}\right)\right)\nonumber\\
& & +\sum_{k=k_{0}+1}^{\infty }E_{\beta }\left(-t^{\beta
}\lambda_{k}\left((-\Delta_{D} )^{\alpha
/2}(I-\Delta_{D})^{\gamma /2}\right)\right)\nonumber\\
& &= M(\beta,\alpha,\gamma,n)+\sum_{k=k_{0}+1}^{\infty }E_{\beta
}\left(-t^{\beta }\lambda_{k}\left((-\Delta_{D} )^{\alpha
/2}(I-\Delta_{D} )^{\gamma /2}\right)\right)\nonumber\\
& & \leq
 M(\beta,\alpha,\gamma,n)+\sum_{k=k_{0}+1}^{\infty }E_{\beta
}\left(-t^{\beta }L_{1}k^{(\alpha +\gamma )/n}\right)
\nonumber\\
& &\leq M(\beta,\alpha,\gamma,n)+\int_{0}^{\infty} E_{\beta
}\left(-t^{\beta }L_{1}x^{(\alpha +\gamma
)/n}\right)dx\nonumber\\
& &= M(\beta,\alpha,\gamma,n)+\frac{t^{-\beta n/(\alpha
+\gamma)}}{(\alpha +\gamma)/n} \int_{0}^{\infty}E_{\beta
}\left(u\right)u^{\frac{n}{\alpha +\gamma }-1}du\nonumber\\
& &\leq M(\beta,\alpha,\gamma,n)+\frac{t^{-\beta n/(\alpha
+\gamma)}}{(\alpha
+\gamma)/n}\int_{0}^{\infty}\frac{u^{\frac{n}{\alpha +\gamma
}-1}}{1+[\Gamma (1+\beta )]^{-1}u}du<\infty ,\nonumber\\
 \label{pfpjj}
\end{eqnarray}

\noindent since
\begin{equation}M(\beta,\alpha,\gamma,n)=\sum_{k=1}^{k_{0}}E_{\beta
}\left(-t^{\beta }\lambda_{k}\left((-\Delta_{D} )^{\alpha
/2}(I-\Delta_{D} )^{\gamma
/2}\right)\right)<\infty,\label{eqp1bb}\end{equation} \noindent and
$\int_{0}^{\infty}\frac{u^{\frac{n}{\alpha +\gamma }-1}}{1+[\Gamma
(1+\beta )]^{-1}u}du<\infty,$ for $\alpha +\gamma
>n.$

\hfill $\blacksquare$
\end{proof}

\bigskip

A mean-square  Gaussian solution, in the weak sense,  to equations
(\ref{1.2})--(\ref{1.2b}) is formulated in Proposition \ref{pr2},
considering $D$ to be a Dirichlet-regular bounded open domain. Note
that, in the classical theory of  boundary value problems, given an
open set $D$ with compact closure $\overline{D}$ in
$\mathbb{R}^{n},$ the classical Dirichlet problem consists of the
extension of  a given continuous function $\psi:
\partial D\longrightarrow \mathbb{R}$   to a continuous function
$\phi : \overline{D} \longrightarrow \mathbb{R}$ such that $\phi $
is harmonic, that is, satisfies the Laplace equation in $D.$ The set
$D$ is termed regular if the Dirichlet problem has a (necessarily
unique) solution for any continuous boundary function $\psi.$ For
example, every simply connected planar domain is regular, but may
have a \emph{bad} boundary, for instance, a fractal boundary (see
Arendt and Schleich \cite{ArendtSchleich09}, pp. 54-55; Fuglede
\cite{Fuglede05}).
  Dirichlet
regularity implies that all the eigenfunctions of the Dirichlet
Laplacian operator on $D$ are  bounded continuous functions on this
domain that vanish continuously on the boundary. This fact will be
exploited in the examples given in Section \ref{examples}, according
to the conditions required on the eigenfunctions, in the derivation
of the main results of this paper.

In a more general setting,   we consider the following definition of
Dirichlet-regular bounded open domain (see, for example, Brelot
\cite{Brelot60},  p. 137 and Theorem 32, and Fuglede
\cite{Fuglede05}, p. 253).
\begin{definition}
\label{def00}
 For a
connected bounded open domain $D$ with boundary $\partial D$ we say
that $\mathbf{x}_{0}\in
\partial D$ is regular if and only if it has a Green kernel $G^{D}$
such that, for each $\mathbf{x}\in D,$
\begin{equation}
\lim_{\mathbf{x}\rightarrow
\mathbf{x}_{0}}G^{D}(\mathbf{x},\mathbf{y})=0,\quad \forall
\mathbf{y}\in D. \label{eqrddgs}
\end{equation}

The set $D$ is regular if every point of $\partial D$ is regular.
\end{definition}

\noindent See also Chen \emph{et al.} \cite{ChenMeerschaert12} for
alternative characterizations of Dirichlet-regular bounded open
domains in terms of the first exit time in the context of
subordinate processes.

The following result provides a mean-square  zero-mean Gaussian
solution, in the weak sense, to the stochastic pseudodifferential
boundary value problem (\ref{1.2})--(\ref{1.2b})  on a
Dirichlet-regular bounded open domain $D.$
\begin{prop}
\label{pr2} Let $c$ be defined as
\begin{equation}
c(t,\mathbf{x})=\int_{0}^{t}\int_{D}G^{D}(t,\mathbf{x};s,\mathbf{y})\varepsilon
(s,\mathbf{y})dsd\mathbf{y},  \label{ris}
\end{equation}
\noindent where $\varepsilon (s,\mathbf{y})$ is space-time zero-mean
Gaussian white noise as given in equation (\ref{1.2}), and

\begin{eqnarray}
&&G^{D}(t,\mathbf{x};s,\mathbf{y})=\nonumber\\
&& =\sum_{k\geq 1}E_{\beta }\left( -\lambda _{k}\left( (-\Delta_{D}
)^{\alpha /2}(I-\Delta_{D} )^{\gamma /2}\right) (t-s)^{\beta
}\right) \phi _{k}(\mathbf{x})\phi _{k}(\mathbf{y}),\ t\geq
s  \nonumber \\
&&G^{D}(t,\mathbf{x};s,\mathbf{y})=0,\ s>t,  \label{sr}
\end{eqnarray}
\noindent with, as before, for each $k\geq 1$ (see Corollary
\ref{cor1})
\begin{equation}
(-\Delta_{D} )^{\alpha /2}(I-\Delta_{D})^{\gamma /2}\phi
_{k}=\lambda _{k}\left( (-\Delta_{D} )^{\alpha /2}(I-\Delta_{D}
)^{\gamma /2}\right) \phi _{k}.  \label{se}
\end{equation}
\noindent Assume that $\{\phi _{k}\{_{k\geq 1}$  are uniformly
bounded by a constant $C(D),$ depending of the geometrical
characteristics of the domain $D,$ i.e., $$C(D)=\sup_{k\geq 1,\
\mathbf{x}\in D}\phi_{k}(\mathbf{x}).$$  Then, for $n<\alpha
+\gamma,$ $c$ in (\ref{ris}) provides a mean-square zero-mean
Gaussian solution to problem (\ref{1.2})--(\ref{1.2b}) on  $D,$ in
the weak-sense in the space $\overline{H}^{\alpha +\gamma }(D).$
Equivalently,
\begin{eqnarray}
& &\int_{D}\left[\frac{\partial ^{\beta }}{\partial t^{\beta
}}c\left( t,\mathbf{x}\right) +\left( -\Delta_{D} \right) ^{\alpha
/2}\left( I-\Delta_{D} \right) ^{\gamma /2}c\left(
t,\mathbf{x}\right)\right]
\psi(\mathbf{x})d\mathbf{x}\nonumber\\
& &  \underset{\mbox{m.s.}}{=} \int_{D}I^{1-\beta }_{t}\varepsilon
\left( t,\mathbf{x}\right)\psi(\mathbf{x})d\mathbf{x},\quad \forall
 \psi \in \overline{H}^{\alpha +\gamma }(D),
\label{wssol}
\end{eqnarray}
\noindent where $\underset{\mbox{m.s.}}{=}$ means the equality in the mean-square sense.
In addition, $c$ has covariance kernel given by, for all $t,s\in \mathbb{R}_{+},$ and $\mathbf{x},\mathbf{y}\in D,$
\begin{equation}
R(t,\mathbf{x};s,\mathbf{y})=E[c(t,\mathbf{x})c(s,%
\mathbf{y})]=\int_{0}^{t\wedge s}\int_{D}G^{D}(t,\mathbf{x};u,%
\mathbf{z})G^{D}(s,\mathbf{y};u,\mathbf{z})dud\mathbf{z}.
\label{covop}
\end{equation}
\end{prop}

\begin{proof}
It is well-known that the solution to the eigenvalue equation
\begin{equation}
\frac{d^{\beta }}{dt^{\beta }}T(t) =-\mu T(t),\quad 0<t\leq T,
 \label{eq2}
\end{equation}
\noindent is given by  the Mittag-Leffler function $E_{\beta }(-\mu
t^{\beta }),$  for any $\mu >0,$  with $E_{\beta }$ being introduced
in equation  (\ref{mlf}). Then, for $\beta \in (0,1),$
  from definition of $G^{D}$ in equations (\ref{sr})--(\ref{se}),
and the definition of the regularized fractional derivative in time
(\ref{1.3}),

\begin{eqnarray}
&& \int_{D}\frac{\partial ^{\beta }}{\partial t^{\beta }}G^{D}\left(
t,\mathbf{x};0,\mathbf{y}\right)\psi(\mathbf{y})d\mathbf{y}=-\int_{D}\sum_{k=1}^{\infty}
\lambda_{k}\left( (-\Delta_{D} )^{\alpha /2}(I-\Delta_{D} )^{\gamma
/2}\right)
\nonumber\\
& & \hspace*{1.5cm}\times  E_{\beta }\left( - \lambda_{k}\left(
(-\Delta_{D} )^{\alpha /2}(I-\Delta_{D} )^{\gamma /2}\right)t^{\beta
}\right) \phi _{k}(\mathbf{x})
\phi_{k}(\mathbf{y})\psi(\mathbf{y})d\mathbf{y}\nonumber\\
& &=-\sum_{k=1}^{\infty} \lambda_{k}\left( (-\Delta_{D} )^{\alpha
/2}(I-\Delta_{D} )^{\gamma /2}\right)E_{\beta }\left( -
\lambda_{k}\left( (-\Delta _{D})^{\alpha /2}(I-\Delta_{D})^{\gamma
/2}\right)t^{\beta }\right)
\nonumber\\
& & \hspace*{2cm}\times \phi
_{k}(\mathbf{x})\int_{D}\phi_{k}(\mathbf{y})\psi(\mathbf{y})d\mathbf{y}
\nonumber\\
& & = -(-\Delta_{D} )^{\alpha /2}(I-\Delta_{D} )^{\gamma
/2}\sum_{k=1}^{\infty} E_{\beta }\left( - \lambda_{k}\left(
(-\Delta_{D} )^{\alpha /2}(I-\Delta _{D})^{\gamma /2}\right)t^{\beta
}\right)\nonumber\\
& & \hspace*{2cm}\times
 \phi_{k}(\mathbf{x}) \psi_{k}=-(-\Delta_{D} )^{\alpha /2}(I-\Delta_{D} )^{\gamma
/2}\mathcal{G}^{D}_{t}\left(\psi \right), \label{vsfe}
\end{eqnarray}
\noindent where $\mathcal{G}^{D}_{t}$ denotes the integral operator
on $L^{2}(D)$ with kernel  $G^{D}\left(
t,\mathbf{x};0,\mathbf{y}\right),$ for each $t>0.$ Note that, from triangle and Cauchy–-Schwarz inequalities,
\begin{eqnarray}
& & \left|\sum_{k=1}^{\infty} \frac{\partial ^{\beta }}{\partial t^{\beta
}} E_{\beta }\left( - \lambda_{k}\left( (-\Delta_{D} )^{\alpha
/2}(I-\Delta_{D})^{\gamma /2}\right)t^{\beta }\right) \phi
_{k}(\mathbf{x})\psi_{k}\right|\nonumber\\ & & \leq  \sum_{k=1}^{\infty}
\lambda_{k}\left( (-\Delta_{D} )^{\alpha /2}(I-\Delta_{D} )^{\gamma
/2}\right)\nonumber\\
& & \hspace*{1.5cm}\times E_{\beta }\left( - \lambda_{k}\left(
(-\Delta_{D} )^{\alpha /2}(I-\Delta _{D})^{\gamma /2}\right)t^{\beta
}\right) |\phi _{k}(\mathbf{x})|\psi_{k}\nonumber\\ & & \leq
C(D)\sum_{k=1}^{\infty} \lambda_{k}\left( (-\Delta _{D})^{\alpha
/2}(I-\Delta_{D} )^{\gamma
/2}\right)\nonumber\\
& & \hspace*{1.5cm}\times E_{\beta }\left( - \lambda_{k}\left(
(-\Delta_{D} )^{\alpha /2}(I-\Delta_{D} )^{\gamma /2}\right)t^{\beta
}\right)\psi_{k}\nonumber\\
& &\leq C(D)\sqrt{\sum_{k=1}^{\infty} [\lambda_{k}\left( (-\Delta _{D})^{\alpha
/2}(I-\Delta_{D} )^{\gamma
/2}\right)\psi_{k}]^{2}}
\nonumber\\
& &
\times \sqrt{\sum_{k=1}^{\infty}[E_{\beta }\left( - \lambda_{k}\left(
(-\Delta_{D} )^{\alpha /2}(I-\Delta_{D} )^{\gamma /2}\right)t^{\beta
}\right)]^{2}}
<\infty,\label{fsd}
\end{eqnarray}
\noindent where, as before
$\psi_{k}=\int_{D}\phi_{k}(\mathbf{y})\psi(\mathbf{y})d\mathbf{y}.$ Here,
$$\sqrt{\sum_{k=1}^{\infty} [\lambda_{k}\left( (-\Delta _{D})^{\alpha
/2}(I-\Delta_{D} )^{\gamma
/2}\right)\psi_{k}]^{2}}<\infty,$$
\noindent since
$\psi \in \overline{H}^{\alpha +\gamma }(D),$   and  $$\sqrt{\sum_{k=1}^{\infty}[E_{\beta }\left( - \lambda_{k}\left(
(-\Delta_{D} )^{\alpha /2}(I-\Delta_{D} )^{\gamma /2}\right)t^{\beta
}\right)]^{2}}<\infty $$ from  Proposition \ref{pr1}, since $\alpha
+\gamma>n.$

 Applying the regularized fractional derivative in
time (\ref{1.3}), from equation  (\ref{vsfe}), we obtain

 \begin{eqnarray}
& &\int_{D} \frac{\partial ^{\beta }}{\partial t^{\beta }}c\left(
t,\mathbf{x}\right)\psi(\mathbf{x})d\mathbf{x}=
\int_{D}\frac{1}{\Gamma (1-\beta )}\frac{d}{dt}\int_{0}^{t}(t-\tau)^{-\beta
}\int_{0}^{\tau}\int_{D}G^{D}\left(
\tau,\mathbf{x};s,\mathbf{y}\right)\nonumber\\
& & \hspace*{5cm}\times \varepsilon
(s,\mathbf{y})\psi(\mathbf{x})d\mathbf{y}dsd\tau d\mathbf{x}\nonumber\\
& & =\int_{D}\frac{1}{\Gamma (1-\beta )}\frac{d}{dt}\int_{0}^{t} u^{-\beta
}\int_{0}^{t-u}\int_{D}G^{D}\left(
t-u,\mathbf{x};s,\mathbf{y}\right)\varepsilon
(s,\mathbf{y})\psi(\mathbf{x})d\mathbf{y} dsdu d\mathbf{x}\nonumber\\
& & =\int_{D}\left(\frac{1}{\Gamma (1-\beta )}\int_{0}^{t}u^{-\beta }\left[\int_{D}G^{D}\left(
t-u,\mathbf{x};t-u,\mathbf{y}\right)\varepsilon
(t-u,\mathbf{y})d\mathbf{y}\right]du\right) \psi(\mathbf{x})d\mathbf{x}\nonumber\\
& &+
\int_{D}\left[\int_{0}^{t-u}\frac{1}{\Gamma (1-\beta )}\int_{D}\left[\frac{d}{dt}\int_{0}^{t}u^{-\beta
}G^{D}\left( t-u,\mathbf{x};s,\mathbf{y}\right)du\right]\varepsilon
(s,\mathbf{y})d\mathbf{y}ds \right]\psi(\mathbf{x})d\mathbf{x}\nonumber\\
& & =\int_{D}\left[\frac{1}{\Gamma (1-\beta )}\int_{0}^{t}u^{-\beta }\varepsilon
(t-u,\mathbf{x})du\right]\psi(\mathbf{x})d\mathbf{x}\nonumber\\
& &+\int_{D}\left[\int_{0}^{t}\int_{D}\frac{\partial ^{\beta
}}{\partial t^{\beta }}G^{D}\left(
t,\mathbf{x};s,\mathbf{y}\right)\varepsilon
(s,\mathbf{y})d\mathbf{y}ds \right]\psi(\mathbf{x})d\mathbf{x}\nonumber\\
& &=\int_{D}I_{t}^{1-\beta }\varepsilon
(t,\mathbf{x})\psi(\mathbf{x})d\mathbf{x}
-\int_{D}\left[(-\Delta_{D} )^{\alpha /2}(I-\Delta _{D})^{\gamma
/2}\int_{0}^{t}\int_{D}G^{D}\left(
t,\mathbf{x};s,\mathbf{y}\right)\right.\nonumber\\
& & \hspace*{5cm} \left.\times \varepsilon
(s,\mathbf{y})d\mathbf{y}ds\right] \psi(\mathbf{x})d\mathbf{x},
\end{eqnarray}
\noindent as we wanted to prove. Here, we have applied that
$$G^{D}\left(
u,\mathbf{x};u,\mathbf{y}\right)=\sum_{k=1}^{\infty}\phi
_{k}(\mathbf{x}) \phi_{k}(\mathbf{y})=\delta
(\mathbf{x}-\mathbf{y}),\quad \forall u>0,$$ \noindent with $\delta
(\mathbf{x}-\mathbf{y})$ denoting the Dirac Delta distribution on
$L^{2}(D)$ such that $$\int_{D}\delta
(\mathbf{x}-\mathbf{y})\varepsilon
(t,\mathbf{y})d\mathbf{y}=\varepsilon (t,\mathbf{x}),$$ \noindent in
the mean-square sense, and in the $L^{2}(D)$-weak sense.

Finally, equation (\ref{covop}) is obtained from straightforward
computation of the covariance function of $c$ in equation
(\ref{ris}), since for $n/2<\alpha +\gamma ,$ $G^{D}$ defines a
Hilbert-Schmidt operator. Consequently, its self-convolution defines
a covariance operator $\mathcal{R}$ in the trace class. Thus, its
covariance kernel $R$ is continuous, and it can be defined pointwise
from equation  (\ref{covop}).
 \hfill $\blacksquare $
\end{proof}

\section{Mean-quadratic local variation in time}
\label{Sec3} \setcounter{section}{4}

 This section provides an upper bound for the mean-quadratic local
 variation of the temporal increments of the mean-square solution
 $c$ defined  in  equation (\ref{ris}) of Proposition
\ref{pr2}. Note that although we have showed in Proposition
\ref{pr2} that $c$ satisfies, in the mean-square sense,  equation
(\ref{1.2}) over the test functions in $\overline{H}^{\alpha +\gamma
}(D),$
 for $\alpha +\gamma >n,$ as we prove, in the following result, equation (\ref{ris}) defines a H\"older continuous, in time,  spatiotemporal
 random field $c,$ under a wider range of parameter $\alpha +\gamma.$ Namely, Theorem \ref{prtqv} below  holds for $\frac{n}{2} <\alpha +\gamma ,$
 and $\beta
<1/2.$

As before,  we will consider the sequence of eigenvalues $$\lambda
_{k}\left( (-\Delta_{D} )^{\alpha /2}(I-\Delta_{D} )^{\gamma
/2}\right) ,\quad k\geq 1,$$\noindent  arranged in
increasing order of their modulus magnitude, with the associated eigenvectors $%
\phi _{k},$ $k\geq 1,$ in the same order.

\begin{theo}
\label{prtqv} Let $c$ be defined as in  (\ref{ris})--(\ref{se}) of
Proposition \ref{pr2}, under the assumption that $C(D)=\sup_{k\geq
1,\ \mathbf{x}\in D}\phi_{k}(\mathbf{x})<\infty.$ Then, for $\beta
<1/2,$ and  $\frac{n}{2} <\alpha +\gamma ,$ the following inequality
holds:
\begin{equation}
E[c(t,\mathbf{x})-c(s,\mathbf{x})]^{2}\leq [C(D)]^{2} g(t-s),
\label{eineqti3}
\end{equation}
\noindent where
\begin{equation}
g(t-s)=\mathcal{O}\left((t-s)^{\left(1-\frac{\beta n}{\alpha
+\gamma}\right)\wedge (1-\beta )}\right), \quad s\rightarrow t, \
0<s<t,\label{eqth1}
\end{equation}

\noindent with $x\wedge y$ denoting the minimum of $x$ and $y,$ for
$x, y\in \mathbb{R}.$
\end{theo}

\begin{proof}
Since $E_{\beta }(-x)$ is a monotone decreasing functions with
values in the interval $[0,1],$ for $x\in \mathbb{R}_{+},$  for
$0<s<t,$ we obtain

\begin{eqnarray}
&
&E[c(t,\mathbf{x})-c(s,\mathbf{x})]^{2}=\int_{0}^{s}\sum_{k=1}^{\infty
} \phi_{k}^{2}(\mathbf{x})\left[E_{\beta
}(-\lambda_{k}\left((-\Delta_{D} )^{\alpha /2}(I-\Delta_{D}
)^{\gamma /2}\right)(t-u)^{\beta })\right. \nonumber\\ &
&\hspace*{1.5cm} \left. -E_{\beta
}\left(-\lambda_{k}\left((-\Delta_{D} )^{\alpha /2}(I-\Delta_{D}
)^{\gamma
/2}\right)(s-u)^{\beta }\right)\right]^{2}du\nonumber\\
& &+\int_{s}^{t}\sum_{k=1}^{\infty }\phi_{k}^{2}(\mathbf{x})E_{\beta
}\left(-2\lambda_{k}\left((-\Delta _{D})^{\alpha /2}(I-\Delta_{D}
)^{\gamma /2}\right)(t-u)^{\beta }\right)du \nonumber\\ &
&\hspace*{1.5cm}\leq [C(D)]^{2}\left[\int_{0}^{s}\sum_{k=1}^{\infty
}\left[E_{\beta }\left(-\lambda_{k}\left((-\Delta_{D} )^{\alpha
/2}(I-\Delta_{D} )^{\gamma /2}\right)(t-u)^{\beta
}\right)\right.\right.
\nonumber\\
& & \left.\left.-E_{\beta }\left((-\lambda_{k}\left((-\Delta_{D}
)^{\alpha /2}(I-\Delta_{D} )^{\gamma /2}\right)(s-u)^{\beta
}\right)\right]^{2}du\right.\nonumber\\ &
&\hspace*{1.5cm}\left.+\int_{s}^{t}\sum_{k=1}^{\infty
}\left[E_{\beta }\left(-\lambda_{k}\left((-\Delta_{D} )^{\alpha
/2}(I-\Delta_{D} )^{\gamma
/2}\right)(t-u)^{\beta }\right)\right]^{2}du\right]\nonumber
\end{eqnarray}
\begin{eqnarray}
& & = [C(D)]^{2}\left[\int_{0}^{s}\sum_{k=1}^{\infty }\left[E_{\beta
}\left(-\lambda_{k}\left((-\Delta_{D} )^{\alpha
/2}(I-\Delta_{D})^{\gamma /2}\right)(t-u)^{\beta
}\right)\right]^{2}\right.\nonumber\\
& &\hspace*{1.5cm}\left.+\left[E_{\beta
}\left((-\lambda_{k}\left((-\Delta_{D} )^{\alpha /2}(I-\Delta_{D}
)^{\gamma /2}\right)(s-u)^{\beta }\right)\right]^{2}\right.\nonumber\\
& &\left. -2E_{\beta }\left(-\lambda_{k}\left((-\Delta_{D} )^{\alpha
/2}(I-\Delta_{D} )^{\gamma /2}\right)(t-u)^{\beta }\right)\right.
\nonumber\\
& &\hspace*{1.5cm} \left.\times E_{\beta
}\left((-\lambda_{k}\left((-\Delta_{D} )^{\alpha
/2}(I-\Delta_{D} )^{\gamma /2}\right)(s-u)^{\beta }\right)du\right.\nonumber\\
& &\left.+\int_{s}^{t}\sum_{k=1}^{\infty }E_{\beta
}\left(-2\lambda_{k}\left((-\Delta_{D} )^{\alpha /2}(I-\Delta_{D}
)^{\gamma /2}\right)(t-u)^{\beta }\right)du\right]\nonumber\\ & &
\leq [C(D)]^{2}\left[\int_{0}^{s}\sum_{k=1}^{\infty }\left[E_{\beta
}\left(-\lambda_{k}\left((-\Delta_{D} )^{\alpha /2}(I-\Delta_{D}
)^{\gamma /2}\right)(t-u)^{\beta }\right)\right]^{2}
\right.\nonumber\\& & \left. +\left[E_{\beta
}\left((-\lambda_{k}\left((-\Delta_{D} )^{\alpha /2}(I-\Delta_{D}
)^{\gamma /2}\right)(s-u)^{\beta }\right)\right]^{2}\right.
\nonumber\\
& & \left.-2\left[E_{\beta }\left(-\lambda_{k}\left((-\Delta_{D}
)^{\alpha /2}(I-\Delta_{D} )^{\gamma /2}\right)(t-u)^{\beta
}\right)\right]^{2}du\right.\nonumber\end{eqnarray}
\begin{eqnarray}
& & \left.+\int_{s}^{t}\sum_{k=1}^{\infty }E_{\beta
}\left(-2\lambda_{k}\left((-\Delta _{D})^{\alpha /2}(I-\Delta_{D}
)^{\gamma /2}\right)(t-u)^{\beta }\right)du\right]\nonumber\\ &
&\leq [C(D)]^{2}\left[\int_{0}^{s}\sum_{k=1}^{\infty }\left[E_{\beta
}\left(-\lambda_{k}\left((-\Delta_{D} )^{\alpha /2}(I-\Delta_{D}
)^{\gamma /2}\right)(s-u)^{\beta
}\right)\right]^{2}du\right.\nonumber \\&
&\left.+\int_{s}^{t}\sum_{k=1}^{\infty }E_{\beta
}\left(-2\lambda_{k}\left((-\Delta_{D} )^{\alpha /2}(I-\Delta_{D}
)^{\gamma /2}\right)(t-u)^{\beta }\right)du\right]. \label{eq11ttt}
\end{eqnarray}

In a similar way to equation (\ref{pfpjj}), from equation
(\ref{eq11ttt}), we obtain
\begin{eqnarray}
& &E[c(t,\mathbf{x})-c(s,\mathbf{x})]^{2}\leq [C(D)]^{2}
\int_{0}^{s}\left[ \widetilde{M}(\beta,\alpha,\gamma,n, (s-u)^{\beta
}) \right.\nonumber\\ & &\left.\hspace*{2cm}+\frac{(s-u)^{-\beta
n/(\alpha +\gamma)}}{(\alpha +\gamma)/n}
\int_{0}^{\infty}\left[E_{\beta
}\left(x\right)\right]^{2}x^{\frac{n}{\alpha +\gamma }-1}dx\right]du \nonumber\\
& & + [C(D)]^{2}\int_{s}^{t}\left[M(\beta,\alpha,\gamma,n,
(t-u)^{\beta})\right. \nonumber\\
& & \hspace*{2cm}\left.+\frac{(t-u)^{-\beta n/(\alpha
+\gamma)}}{(\alpha +\gamma)/n} \int_{0}^{\infty}E_{\beta
}\left(2x\right)x^{\frac{n}{\alpha +\gamma }-1}dx\right]du \nonumber
\end{eqnarray}
\begin{eqnarray}
& & \leq [C(D)]^{2}\int_{0}^{s}\left[
\widetilde{M}(\beta,\alpha,\gamma,n,(s-u)^{\beta })
\right.\nonumber\\& & \hspace*{2cm}\left.+\frac{(s-u)^{-\beta
n/(\alpha +\gamma)}}{(\alpha
+\gamma)/n}\int_{0}^{\infty}\frac{x^{\frac{n}{\alpha +\gamma
}-1}}{(1+[\Gamma (1+\beta )]^{-1}x)^{2}}dx\right]du\nonumber\\
& & +[C(D)]^{2}\int_{s}^{t}\left[M(\beta,\alpha,\gamma,n,
(t-u)^{\beta})\right.\nonumber\\& &
\hspace*{2cm}\left.+\frac{(t-u)^{-\beta n/(\alpha +\gamma)}}{(\alpha
+\gamma)/n} \int_{0}^{\infty}\frac{x^{\frac{n}{\alpha +\gamma
}-1}}{1+[\Gamma (1+\beta )]^{-1}2x}dx\right]du,\nonumber\\
\label{ineqfund}
\end{eqnarray}

\noindent where
\begin{eqnarray}
& &M(\beta,\alpha,\gamma,n,
(s-u)^{\beta})=\sum_{k=1}^{k_{0}}E_{\beta }\left(-2(s-u)^{\beta
}\lambda_{k}\left((-\Delta_{D} )^{\alpha /2}(I-\Delta_{D} )^{\gamma
/2}\right)\right)
\nonumber\\
& & \widetilde{M}(\beta,\alpha,\gamma,n, (s-u)^{\beta
})=\sum_{k=1}^{k_{0}}\left[E_{\beta }\left(-(s-u)^{\beta
}\lambda_{k}\left((-\Delta_{D} )^{\alpha /2}(I-\Delta_{D} )^{\gamma
/2}\right)\right)\right]^{2}.\nonumber\\\label{eqp1bbcc}\end{eqnarray}

Hence,  \begin{eqnarray}&
&E[c(t,\mathbf{x})-c(s,\mathbf{x})]^{2}\leq [C(D)]^{2}\left[
K_{1}(\beta,\alpha,\gamma,n)s^{1-2\beta}+K_{2}(\beta,\alpha,\gamma,n)s^{1-\beta
n/(\alpha +\gamma )}\right. \nonumber\\
& &
\left.+K_{3}(\beta,\alpha,\gamma,n)(t-s)^{1-\beta}+K_{4}(\beta,\alpha,\gamma,n)(t-s)^{1-\beta
n/(\alpha +\gamma )}\right]. \label{ineqfundb}
\end{eqnarray}

Thus, when $s\rightarrow t,$ $s<t,$ we have
$$E[c(t,\mathbf{x})-c(s,\mathbf{x})]^{2}\leq [C(D)]^{2}g(t-s),$$

\noindent with \begin{eqnarray} g(t-s)&=& \left[
K_{1}(\beta,\alpha,\gamma,n)s^{1-2\beta}+K_{2}(\beta,\alpha,\gamma,n)s^{1-\beta
n/(\alpha +\gamma )}\right. \nonumber\\
& &
\left.+K_{3}(\beta,\alpha,\gamma,n)(t-s)^{1-\beta}+K_{4}(\beta,\alpha,\gamma,n)(t-s)^{1-\beta
n/(\alpha +\gamma )}\right]\nonumber\\
& &= \mathcal{O}\left((t-s)^{1-\left(\beta n/(\alpha +\gamma
)\right)\wedge (1-\beta )}\right).\nonumber\end{eqnarray}

\hfill $\blacksquare $
\end{proof}
\section{Mean-quadratic local variation in space}
\label{Sec4}

\setcounter{section}{5}

 The fractional local asymptotic exponent, in the mean-square sense,  of the spatial increments of the solution $c,$ derived in Proposition  \ref{pr2}, is obtained
in the following result.
\begin{theo}
\label{th2} Let $c$ be as given in equations
(\ref{ris})--(\ref{se}), for  $n<\alpha +\gamma.$ Assume that for
every $k\geq 1,$
$$|\phi_{k}(\mathbf{x}+\mathbf{h})-\phi_{k}(\mathbf{x})|=\mathcal{O}(\|\mathbf{h}\|^{\Upsilon
}),\quad \|\mathbf{h}\|\rightarrow 0,\quad \Upsilon>0.$$ In
particular, for each $k\geq 1,$ and for $\|\mathbf{h}\|$  small,
$$|\phi_{k}(\mathbf{x}+\mathbf{h})-\phi_{k}(\mathbf{x})|\leq
C_{k}\|\mathbf{h}\|^{\Upsilon},$$ \noindent for certain positive
constant $C_{k}.$ If $\sup_{k}C_{k}=C<\infty,$ then, for each $t>0,$
\begin{equation}
E[c(t,\mathbf{x})-c(t,\mathbf{y})]^{2}=\mathcal{O}(\|\mathbf{x}-\mathbf{y}\|^{2\Upsilon
}),\quad \|\mathbf{x}-\mathbf{y}\|\rightarrow 0,\quad \Upsilon>0.
 \label{sicc}
\end{equation}

\noindent Thus, for  $\Vert \mathbf{x}-\mathbf{y}\Vert $
sufficiently small,
\begin{equation}
E[c(t,\mathbf{x})-c(t,\mathbf{y})]^{2}\leq
Cg(t)\|\mathbf{x}-\mathbf{y}\|^{2\Upsilon }, \label{sicc2}
\end{equation}%
\noindent where
\begin{eqnarray}
g(t)&=&t^{1-\beta}\sum_{k=1}^{\infty }\frac{\Gamma (1+\beta
)}{\lambda_{k}\left((-\Delta )^{\alpha /2}_{D}(I-\Delta )^{\gamma
/2}_{D}\right)} ,\quad t>0.\nonumber
\end{eqnarray}

\end{theo}

\begin{proof}

  Applying   H\"older continuity of the eigenvectors,
 from Lemma
\ref{th4s},  we have, for every $t>0,$
\begin{eqnarray}
& &E[c(t,\mathbf{x})-c(t,\mathbf{y})]^{2}
\nonumber\\
& & =\int_{0}^{t}\sum_{k=1}^{\infty } \left[E_{\beta
}(-\lambda_{k}\left((-\Delta )^{\alpha /2}_{D}(I-\Delta )^{\gamma
/2}_{D}\right)(t-u)^{\beta
})\phi_{k}(\mathbf{x})\right.\nonumber\\
& & \left. -E_{\beta }(-\lambda_{k}\left((-\Delta )^{\alpha
/2}_{D}(I-\Delta )^{\gamma /2}_{D}\right)(t-u)^{\beta
})\phi_{k}(\mathbf{y})\right]^{2}du\nonumber\\
& &=\int_{0}^{t}\sum_{k=1}^{\infty }
[\phi_{k}(\mathbf{x})-\phi_{k}(\mathbf{y})]^{2}\left[E_{\beta
}(-\lambda_{k}\left((-\Delta )^{\alpha /2}_{D}(I-\Delta )^{\gamma
/2}_{D}\right)(t-u)^{\beta })\right]^{2}du\nonumber\\
& & \leq C\|\mathbf{x}-\mathbf{y}\|^{2\Upsilon
}\int_{0}^{t}\sum_{k=1}^{\infty }\left[E_{\beta
}(-\lambda_{k}\left((-\Delta )^{\alpha /2}_{D}(I-\Delta )^{\gamma
/2}_{D}\right)(t-u)^{\beta })\right]^{2}du\nonumber\\
& & \leq C\|\mathbf{x}-\mathbf{y}\|^{2\Upsilon
}\int_{0}^{t}\sum_{k=1}^{\infty }\frac{\Gamma (1+\beta
)}{\lambda_{k}\left((-\Delta )^{\alpha /2}_{D}(I-\Delta )^{\gamma
/2}_{D}\right)\nu^{\beta }}d\nu \nonumber\\
& & =C\|\mathbf{x}-\mathbf{y}\|^{2\Upsilon
}t^{1-\beta}\sum_{k=1}^{\infty }\frac{\Gamma (1+\beta
)}{\lambda_{k}\left((-\Delta )^{\alpha /2}_{D}(I-\Delta )^{\gamma
/2}_{D}\right)}
=Cg(t)\|\mathbf{x}-\mathbf{y}\|^{2\Upsilon},\nonumber\\
\end{eqnarray}
\noindent \noindent as we wanted to prove.  Note that,  from
Corollary \ref{cor1}, for each  fixed $t>0,$
$$g(t)=t^{1-\beta}\sum_{k=1}^{\infty }\frac{\Gamma (1+\beta
)}{\lambda_{k}\left((-\Delta )^{\alpha /2}_{D}(I-\Delta )^{\gamma
/2}_{D}\right)} <\infty ,$$ \noindent for $\alpha +\gamma >n.$
\end{proof}

\section{Mean quadratic local variation in time and space}
\label{Sec5}

\setcounter{section}{6}

In this section we apply the results derived in Theorems \ref{prtqv}
and \ref{th2} to obtain the mean-quadratic local variation
properties of the spatiotemporal increments of the weak-sense
solution $c$ to equations (\ref{1.2})--(\ref{1.2b}).

\begin{theo}
\label{stinc} Under conditions of Theorems \ref{prtqv} and
\ref{th2}, let $c$ be  defined in equations (\ref{ris})--(\ref{se}).
Then,  as $s\rightarrow t,$ $s,t\in (0,T],$ and $\Vert
\mathbf{x}-\mathbf{y}\Vert \rightarrow 0,$
\begin{eqnarray*}
&&E[c(t,\mathbf{x})-c(s,\mathbf{y})]^{2} \\
&\leq &\widetilde{C}(D,T,\beta ,\alpha,\gamma,\Upsilon
,n)\|(t,\mathbf{x})-(s,\mathbf{y})\|^{\left(1-\frac{\beta n}{\alpha
+\gamma}\right)\wedge (1-\beta )\wedge 2\Upsilon} ,
\end{eqnarray*}
\noindent where $$\widetilde{C}(D,T,\beta ,\alpha,\gamma,\Upsilon,
n)=8([C(D)]^{2}\vee
Cg(T))\left(\frac{1}{2}\right)^{1/2\left[\left(1-\frac{\beta
n}{\alpha +\gamma}\right)\wedge (1-\beta )\wedge 2\Upsilon\right].
},$$ \noindent with  $x\vee y$ representing the maximum of $x$ and
$y,$ and, as before,  $x\wedge y$ representing the minimum. Here,
$C(D)$ is introduced in Theorem \ref{prtqv}, and $C$ and $g(T)$ are
given in Theorem \ref{th2}.
\end{theo}

\begin{proof}
The proof follows from Theorems \ref{prtqv} and \ref{th2}.
Specifically, from equations (\ref{eineqti3})--(\ref{eqth1}), and
(\ref{sicc})--(\ref{sicc2}), as $s\rightarrow t,$ and
$\|\mathbf{x}-\mathbf{y}\|\rightarrow 0,$
\begin{eqnarray}
& &E[c(t,\mathbf{x})-c(s,\mathbf{y})]^{2}
=E[c(t,\mathbf{x})-c(s,\mathbf{x})+c(s,\mathbf{x})-c(s,\mathbf{y})]^{2}\nonumber\\
&
&=E[c(t,\mathbf{x})-c(s,\mathbf{x})]^{2}+E[c(s,\mathbf{x})-c(s,\mathbf{y})]^{2}\nonumber\\
& & \hspace*{2cm}+
2E\left[(c(t,\mathbf{x})-c(s,\mathbf{x}))(c(s,\mathbf{x})-c(s,\mathbf{y}))\right]\nonumber\\
& & \leq
E[c(t,\mathbf{x})-c(s,\mathbf{x})]^{2}+E[c(s,\mathbf{x})-c(s,\mathbf{y})]^{2}\nonumber\\
& & \hspace*{2cm}+
2\left|E\left[(c(t,\mathbf{x})-c(s,\mathbf{x}))(c(s,\mathbf{x})-c(s,\mathbf{y}))\right]\right|\nonumber
\end{eqnarray}
\begin{eqnarray}
& &\leq
E[c(t,\mathbf{x})-c(s,\mathbf{x})]^{2}+E[c(s,\mathbf{x})-c(s,\mathbf{y})]^{2}\nonumber\\
& &
\hspace*{2cm}+2\sqrt{E\left[(c(t,\mathbf{x})-c(s,\mathbf{x}))^{2}\right]}
\sqrt{E\left[(c(s,\mathbf{x})-c(s,\mathbf{y}))^{2}\right]}\nonumber\\
& & \leq [C(D)]^{2}|t-s|^{\left(1-\frac{\beta n}{\alpha
+\gamma}\right)\wedge (1-\beta
)}+Cg(s)\|\mathbf{x}-\mathbf{y}\|^{2\Upsilon
}\nonumber\\
& & \hspace*{2cm}+2\sqrt{[C(D)]^{2}|t-s|^{\left(1-\frac{\beta
n}{\alpha +\gamma}\right)\wedge
(1-\beta )}Cg(s)\|\mathbf{x}-\mathbf{y}\|^{2\Upsilon }}\nonumber\\
& & \leq [C(D)]^{2}|t-s|^{\left(1-\frac{\beta n}{\alpha
+\gamma}\right)\wedge (1-\beta
)}+Cg(s)\|\mathbf{x}-\mathbf{y}\|^{2\Upsilon
}\nonumber\\
& & \hspace*{2cm}+2\sqrt{\left([C(D)]^{2}|t-s|^{\left(1-\frac{\beta
n}{\alpha +\gamma}\right)\wedge (1-\beta
)}\right)^{2}+\left(Cg(s)\|\mathbf{x}-\mathbf{y}\|^{2\Upsilon
}\right)^{2}} \nonumber \\ & & \leq 2([C(D)]^{2}\vee
Cg(T))\left[|t-s|^{\left(1-\frac{\beta n}{\alpha
+\gamma}\right)\wedge (1-\beta )}\right.
\nonumber\\
& & \hspace*{2cm} \left.+\|\mathbf{x}-\mathbf{y}\|^{2\Upsilon
}+\sqrt{\left(|t-s|^{\left(1-\frac{\beta n}{\alpha
+\gamma}\right)\wedge (1-\beta
)}\right)^{2}+\left(\|\mathbf{x}-\mathbf{y}\|^{2\Upsilon
}\right)^{2}}\right]\nonumber\\ & & \leq 4([C(D)]^{2}\vee
Cg(T))\left[\frac{1}{2}\left(|t-s|^{\left(1-\frac{\beta n}{\alpha
+\gamma}\right)\wedge (1-\beta )\wedge 2\Upsilon }\right.\right.
\nonumber\\ & & \hspace*{2cm} \left.\left.
+\|\mathbf{x}-\mathbf{y}\|^{\left(1-\frac{\beta n}{\alpha
+\gamma}\right)\wedge (1-\beta )\wedge 2\Upsilon
}\right)\right.\nonumber\end{eqnarray}
\begin{eqnarray}
& & \left.+\sqrt{\frac{1}{4}\left(|t-s|^{\left(1-\frac{\beta
n}{\alpha +\gamma}\right)\wedge (1-\beta )\wedge 2\Upsilon
}\right)^{2} +\left(\|\mathbf{x}-\mathbf{y}\|^{\left(1-\frac{\beta
n}{\alpha +\gamma}\right)\wedge (1-\beta )\wedge 2\Upsilon
}\right)^{2}}\right].\nonumber\\\label{jensenineq}
\end{eqnarray}

Under the conditions assumed, $0<\left(1-\frac{\beta n}{\alpha
+\gamma}\right)\wedge (1-\beta )\wedge 2\Upsilon <1,$ hence, we can
apply Jensen's inequality for concave function $x^{\xi},$ $0<\xi<1,$
obtaining from the last inequality in (\ref{jensenineq}),
\begin{eqnarray}
& &E[c(t,\mathbf{x})-c(s,\mathbf{y})]^{2} \leq 4([C(D)]^{2}\vee
Cg(T)) \nonumber\\& &
\times\left[\left(\frac{1}{2}\right)^{1/2\left[\left(1-\frac{\beta
n}{\alpha +\gamma}\right)\wedge (1-\beta )\wedge 2\Upsilon\right]
}\left(|t-s|^{2}\right. \right.\nonumber\\& &
\left.\left.\hspace*{4cm}+\|\mathbf{x}-\mathbf{y}\|^{2}\right)^{1/2\left[\left(1-\frac{\beta
n}{\alpha +\gamma}\right)\wedge
(1-\beta )\wedge 2\Upsilon \right]}\right.\nonumber\\
& & \left.+\sqrt{\left(\frac{1}{4}\right)^{\left(1-\frac{\beta
n}{\alpha +\gamma}\right)\wedge (1-\beta )\wedge
2\Upsilon}\left(|t-s|^{2}
+\|\mathbf{x}-\mathbf{y}\|^{2}\right)^{\left(1-\frac{\beta n}{\alpha
+\gamma}\right)\wedge (1-\beta )\wedge 2\Upsilon}}\right]\nonumber
\end{eqnarray}
\begin{eqnarray}
& &\leq 4([C(D)]^{2}\vee
Cg(T))\left(\frac{1}{2}\right)^{1/2\left[\left(1-\frac{\beta
n}{\alpha +\gamma}\right)\wedge (1-\beta )\wedge 2\Upsilon\right]
}\nonumber\\& & \hspace*{4cm}\times
\left[\left(|t-s|^{2}+\|\mathbf{x}-\mathbf{y}\|^{2}\right)^{1/2\left[\left(1-\frac{\beta
n}{\alpha +\gamma}\right)\wedge
(1-\beta )\wedge 2\Upsilon \right]}\right.\nonumber\\
& & +\left.\sqrt{\left(|t-s|^{2}
+\|\mathbf{x}-\mathbf{y}\|^{2}\right)^{\left(1-\frac{\beta n}{\alpha
+\gamma}\right)\wedge (1-\beta )\wedge 2\Upsilon}}\right]\nonumber\\
& &  =2C(D,T,\beta ,\alpha,\gamma,\Upsilon,n
)\|(t,\mathbf{x})-(s,\mathbf{y})\|^{\left(1-\frac{\beta n}{\alpha
+\gamma}\right)\wedge (1-\beta )\wedge 2\Upsilon},
\end{eqnarray}
\noindent as we wanted to prove, with $\widetilde{C}(D,T,\beta
,\alpha,\gamma,\Upsilon ,n)= 2C(D,T,\beta ,\alpha,\gamma,\Upsilon,n
),$ and $$C(D,T,\beta ,\alpha,\gamma,\Upsilon, n)=4([C(D)]^{2}\vee
Cg(T))\left(\frac{1}{2}\right)^{1/2\left[\left(1-\frac{\beta
n}{\alpha +\gamma}\right)\wedge (1-\beta )\wedge 2\Upsilon\right]
}.$$

\hfill $\blacksquare $
\end{proof}
\subsection{Sample-path properties}

 The following result provides the sample path local regularity
properties of the mean-square weak-sense Gaussian  solution $c$ to
equations (\ref{1.2})--(\ref{1.2b}). Note that the derived Gaussian
solution $c$ is not fractional differentiable in time in the
strong-sense (see Proposition \ref{pr2}). However, it is H\"older
continuous, in the mean square sense, under the conditions assumed
in the previous sections.
\begin{theo}
\label{thspp} Let $c$ be  defined in equations
(\ref{ris})--(\ref{se}). Under the conditions of Theorems
\ref{prtqv} and
\ref{th2}, with probability one, the following inequalities hold, as $s\rightarrow t,$ and $\Vert \mathbf{x}-\mathbf{y%
}\Vert \rightarrow 0,$
\begin{eqnarray}
&&\sup_{|t-s|<\delta }|c(t,\mathbf{x})-c(s,\mathbf{x})|^{2}  \notag \\
&\leq &Z\delta ^{\left(1-\frac{\beta n}{\alpha +\gamma}\right)\wedge
(1-\beta )}+H_{1}\delta ^{\left(1-\frac{\beta n}{\alpha
+\gamma}\right)\wedge (1-\beta )}\left[ \log \left( \frac{1}{\delta
}\right) \right] ^{1/2}  \notag
\\
&&\sup_{\Vert \mathbf{x}-\mathbf{y}\Vert <\delta }|c(t,\mathbf{x})-c(t,%
\mathbf{y})|^{2}  \notag
\end{eqnarray}
\begin{eqnarray}
&\leq &Y\delta ^{2\Upsilon}+H_{2}\delta ^{2\Upsilon}
\left[ \log \left( \frac{1}{\delta }\right) \right] ^{1/2}  \notag \\
&&\sup_{\Vert (t,\mathbf{x})-(s,\mathbf{y})\Vert <\delta }|c(t,\mathbf{x}%
)-c(s,\mathbf{y})|^{2}  \notag \\
&\leq &X\delta ^{\left(1-\frac{\beta n}{\alpha
+\gamma}\right)\wedge (1-\beta )\wedge 2\Upsilon}  \notag \\
&&+H_{3}\delta ^{\left(1-\frac{\beta n}{\alpha +\gamma}\right)\wedge
(1-\beta )\wedge 2\Upsilon}\left[ \log \left( \frac{1}{\delta
}\right) \right] ^{1/2},
\notag \\
&&
\end{eqnarray}%
\noindent where $Z,Y$ and $X$ are positive random variables, and
$H_{i},$
 $i=1,2,3,$ are positive constants that could depend on the geometrical
characteristics of the domain $D$ considered, like the boundary.
\end{theo}

 The proof directly follows from Theorems \ref{prtqv}--\ref{stinc},
and Theorem 3.3.3 in Adler \cite{Adler81}, p.57.

\section{Examples}
\label{examples}

\setcounter{section}{7}
 In the following subsections we consider some special cases of domain $D,$ where the
derived results can be applied.  Specifically, in the examples
introduced  below,  the  eigenfunctions of the Dirichlet negative
Laplacian operator can be explicitly computed, and Theorem
\ref{theorem:1} allows us to define the weak-sense mean-square
solution to  (\ref{1.2})--(\ref{1.2b}), in terms of such
eigenfunctions  as given in equations (\ref{ris})--(\ref{sr}) in
Proposition \ref{pr2} (see, for example, Grebenkov and Nguyen
\cite{Grebenkov13}).  The  conditions required in Theorems
\ref{prtqv}, \ref{th2}, \ref{stinc} and \ref{thspp}, for  the
continuity of $c$ in the mean-square sense and in the sample path
sense,  are also verified.

\subsection{Intervals, rectangles, parallelepipeds}

Let us consider the case where $D=(0,L_{1}),\dots ,(0,L_{n})\subset \mathbb{R%
}^{n},$ where $L_{i}>0,$ for $i=1,\dots ,n,$ the method of
separation of variables
 yields the following eigenvectors for the Dirichlet negative Laplacian operator:
\begin{equation}
\phi _{k_{1},\dots ,k_{n}}(x_{1},\dots
,x_{n})=\prod_{i=1}^{n}\sin\left( \frac{\pi
(k_{i}+1)x_{i}}{L_{i}}\right) , \quad  (k_{1},\dots ,k_{n})\in
\mathbb{N}_{*}^{n},\label{eigv}
\end{equation}%
\noindent with $\lambda _{k_{1},\dots ,k_{n}}=\lambda _{k_{1}}+\dots
,+\lambda _{k_{n}},$ and $\lambda _{k_{i}}=\frac{\pi ^{2}(k_{i}+1)^{2}}{%
L_{i}^{2}},$ for $i=1,\dots ,n.$ In this case, the fundamental
solution to the fractional space-time pseudodifferential equation
\begin{equation}
\frac{\partial ^{\beta }c}{\partial t^{\beta
}}(t,\mathbf{x})=-(-\Delta_{D} )^{\alpha /2}(I-\Delta_{D})^{\gamma
/2} c(t,\mathbf{x})  \label{fhb}
\end{equation}%
\begin{equation*}
t\in \mathbb{R}_{+},\quad \mathbf{x}\in D=(0,L_{1}),\dots ,(0,L_{n})
\end{equation*}%
\noindent is given, for $\alpha +\gamma >n,$  by
\begin{eqnarray}
&& G_{t-s}(x_{1},\dots ,x_{n};y_{1},\dots ,y_{n})
=\sum_{(k_{1},\dots k_{n})\in \mathbb{N}_{\ast }^{n}}E_{\beta
}\left( -\left(\sum_{i=1}^{n}\frac{\pi
^{2}(k_{i}+1)^{2}}{L_{i}^{2}}\right)^{\alpha /2}\right.\nonumber\\
&&\left.\hspace*{5cm}\times  \left(1+\left(\sum_{i=1}^{n}\frac{\pi
^{2}(k_{i}+1)^{2}}{L_{i}^{2}}\right)\right)^{\gamma /2}(t-s)^{\beta }\right)   \nonumber\\
& &\hspace*{2cm}\times\left[ \prod_{i=1}^{n}\sin\left( \frac{\pi (k_{i}+1)x_{i}}{L_{i}}%
\right) \right] \left[ \prod_{i=1}^{n}\sin\left( \frac{\pi (k_{i}+1)y_{i}}{%
L_{i}}\right) \right] ,\quad t\geq s   \nonumber\\
&&\hspace*{2cm}\mbox{and}\quad G(t,\mathbf{x};s,\mathbf{y})=0,\quad
s>t.  \nonumber
\end{eqnarray}%
From H\"{o}lder's inequality (see, for example, equation (6.4) in
Grebenkov and Nguyen \cite{Grebenkov13}),
\begin{equation*}
|\phi _{k_{1},\dots ,k_{n}}(x_{1},\dots ,x_{n})|\leq \sqrt{%
\prod_{i=1}^{n}L_{i}}\frac{\prod_{i=1}^{n}L_{i}}{2^{n/2+n}\pi
^{n/2}}.
\end{equation*}%
Therefore, in the previous computations in Theorems \ref{prtqv}, for
$\beta <1/2,$   we consider
\begin{equation*}
C(D)=\frac{\left[ \prod_{i=1}^{n}L_{i}\right] ^{3/2}}{2^{n/2+n}\pi
^{n/2}}.
\end{equation*}
\noindent Theorem \ref{th2} also holds, since
$\prod_{i=1}^{n}\sin\left( \frac{\pi (k_{i}+1)x_{i}}{L_{i}}\right)
,$ $(k_{1},\dots ,k_{n})\in \mathbb{N}_{*}^{n},$ are continuously
differentiable, and hence, H\"older continuous. Theorems \ref{stinc}
and \ref{thspp} then follow for $\beta <1/2,$ and, as before, for
$\alpha +\gamma >n.$

\subsection{Balls}

 Let us consider equations  (\ref{1.2})--(\ref{1.2b}) on the
ball. That is,
\begin{equation}
\frac{\partial ^{\beta }c}{\partial t^{\beta
}}(t,\mathbf{x})=-(-\Delta_{D} )^{\alpha /2}(I-\Delta_{D})^{\gamma
/2} c(t,\mathbf{x}) ,  \label{fhb}
\end{equation}%
\begin{equation*}
t\in \mathbb{R}_{+},\quad \mathbf{x}\in D=\{\mathbf{x}\in
\mathbb{R}^{n},\quad \Vert \mathbf{x}\Vert
<R\}=\mathcal{B}_{R}(\mathbf{0}),\ R>0.
\end{equation*}%
\noindent For $\alpha +\gamma >n,$ the Green function is defined as
\begin{eqnarray}
&&\hspace*{-1cm}G_{t-s}(\rho ,\theta ;\rho ^{\prime },\theta
^{\prime }) =\sum_{l=0}^{\infty }\sum_{r=1}^{\infty
}\sum_{m=1}^{h(n,l)}E_{\beta }\left(-(\mu _{l,r})^{\alpha /2}(1+\mu
_{l,r})^{\gamma /2}(t-s)^{\beta }\right)
\nonumber \\
&&\hspace*{3.5cm}\times \phi _{l,r,m}(\rho ,\theta )\phi
_{l,r,m}(\rho
^{\prime },\theta ^{\prime }),\quad t\geq s  \nonumber \\
&&\mbox{and}\quad G(t,\mathbf{x};s,\mathbf{y})=0,\quad s>t,
\nonumber
\end{eqnarray}%
\noindent where we have considered the spherical coordinates $\mathbf{x}%
=(\rho ,\theta ),$ $\mathbf{y}=(\rho ^{\prime },\theta ^{\prime }),$
and, as before, $E_{\beta }$ is the Mittag-Leffler function.
Moreover,
\begin{equation*}
(-\Delta )_{D}\phi _{l,r,m}(\rho ,\theta )=\mu _{l,r}\phi
_{l,r,m}(\rho ,\theta ),
\end{equation*}%
\noindent for $l\in \mathbb{N},$ with $m=1,2,\dots ,h(n,l),$ $r\in \mathbb{N}%
_{\ast }$ and $h(n,l)=\frac{(2l+n-2)(l+n-3)!}{(n-2)!l!},$ the number
of spherical harmonics. Here, the eigenvalues are given by $\mu
_{l,r}=\left( \frac{\xi _{l,r}^{2}}{R^{2}}\right) ,$ and the
eigenvectors $\phi
_{l,r,m}(\rho ,\theta )=c_{l,r,m}\mathcal{J}_{l+\frac{n-2}{2}}\left( \xi _{l+%
\frac{n-2}{2},r}\frac{\rho }{R}\right) \mathcal{S}_{l,m}(\theta )$
are
defined in terms of the Bessel function of the first kind of order $\nu ,$ $%
\mathcal{J}_{\nu },$ and the orthonormal spherical harmonics on the
sphere of radius one, $\mathcal{S}_{l,m}(\theta ).$ Note that
$c_{l,r,m}$ is the
normalizing constant, and $\xi _{\nu ,r}$ is the $r$th positive root of $%
\mathcal{J}_{\nu }.$

From H\"{o}lder inequality, since $\{\phi _{l,r,m}\}$ are normalized
in the
space of square integrable functions over the ball of radius $R,$ i.e.., $%
\Vert \phi _{l,r,m}\Vert _{2}=1,$ we obtain
\begin{equation*}
|\phi _{l,r,m}(\rho ,\theta )|\leq \Vert \phi _{l,r,m}\Vert _{1}\leq
\Vert
\phi _{l,r,m}\Vert _{2}\sqrt{\int_{\mathcal{B}_{R}(\mathbf{0})}d\mathbf{x}}=%
\sqrt{\frac{\pi ^{n/2}R^{n}}{\Gamma \left( \frac{n}{2}+1\right) }}
\end{equation*}%
\noindent (see, for example, equation (6.4) in Grebenkov and Nguyen
\cite{Grebenkov13}). Thus, in the previous computations in Theorem
\ref{prtqv}, we can consider
$C(D)=|\mathcal{B}_{R}(\mathbf{0})|^{1/2}.$ Theorem \ref{th2} also
holds, since Bessel functions of the first kind and order $\nu ,$ on
a closed interval,   and the orthonormal spherical harmonics on the
sphere of radius one, are H\"older continuous. Theorems \ref{stinc}
and \ref{thspp} then follow for $\beta <1/2,$ and, as before, for
$\alpha +\gamma >n.$

 \subsubsection*{Circular annulus}

 Let us now consider, for $\alpha +\gamma >n,$ and $\beta <1/2,$
\begin{equation}
\frac{\partial ^{\beta }c}{\partial t^{\beta
}}(t,\mathbf{x})=-(-\Delta_{D} )^{\alpha /2}(I-\Delta_{D})^{\gamma
/2} c(t,\mathbf{x}) ,  \label{fhb}
\end{equation}%
\begin{equation*}
t\in \mathbb{R}_{+},\quad \mathbf{x}\in D=\{(x_{1},x_{2})\in
\mathbb{R}^{2};\quad R_{0}<|x|<R\}.
\end{equation*}%
In polar coordinates, $x_{1}=r\cos \varphi ,$ $x_{2}=r\sin \varphi
,$ the Laplace operator $\Delta $ admits the expression
\begin{equation*}
\Delta =\frac{\partial ^{2}}{\partial r^{2}}+\frac{1}{r}\frac{\partial }{%
\partial r}+\frac{1}{r^{2}}\frac{\partial ^{2}}{\partial \varphi ^{2}}.
\end{equation*}

 The fundamental solution (in polar coordinates) of
$$\frac{\partial ^{\beta }c}{\partial t^{\beta
}}(t,\mathbf{x})=-(-\Delta_{D} )^{\alpha /2}(I-\Delta_{D})^{\gamma
/2} c(t,\mathbf{x}),$$

\noindent with $$-\Delta_{D}=-\left(\frac{\partial ^{2}}{\partial r^{2}}+\frac{1}{r}\frac{\partial }{%
\partial r}+\frac{1}{r^{2}}\frac{\partial ^{2}}{\partial \varphi ^{2}}\right),$$ \noindent  is then given by
\begin{eqnarray}
& &G_{t-s}(r,\varphi ,r^{\prime },\varphi ^{\prime
})=\sum_{n=0}^{\infty }\sum_{k=1}^{\infty }\sum_{l=1}^{2}E_{\beta
}\left( -\left(\alpha _{n,k}^{2}/R^{2}\right)^{\alpha
/2}\left(1+\left( \alpha _{n,k}^{2}/R^{2}\right)\right)^{\gamma
/2}(t-s)^{\beta }\right) \nonumber\\& & \hspace*{5cm} \times
u_{n,k,l}(r,\varphi )u_{n,k,l}(r^{\prime },\varphi ^{\prime
}),\nonumber
\end{eqnarray}

 \noindent where
\begin{equation*}
u_{n,k,l}(r,\varphi )=J_{n}(r\alpha _{n,k}/R)+c_{n,k}Y_{n}(r\alpha
_{n,k}/R)\times \left\{
\begin{array}{l}
\cos (n\varphi ),\ l=1 \\
\sin (n\varphi ),\ l=2 \ (n\neq 0),%
\end{array}%
\right.
\end{equation*}
 \noindent with $J_{n}$ and $Y_{n}$ being the Bessel functions of
the first and second kind, and the coefficients $\alpha _{n,k}$ and
$c_{n,k}$ being  set by the boundary conditions at $r=R_{0},$ and
$r=R.$
\begin{eqnarray*}
0 &=&\frac{\alpha _{n,k}}{R}\left[ J_{n}^{\prime }(\alpha
_{n,k})+c_{n,k}Y_{n}^{\prime }(\alpha _{n,k})\right] +h\left[
J_{n}(\alpha
_{n,k})+c_{n,k}Y_{n}(\alpha _{n,k})\right]  \\
0 &=&-\frac{\alpha _{n,k}}{R}\left[ J_{n}^{\prime }\left( \alpha _{n,k}\frac{%
R_{0}}{R}\right) +c_{n,k}Y_{n}^{\prime }\left( \alpha _{n,k}\frac{R_{0}}{R}%
\right) \right]
\nonumber\\
 &+&h\left[ J_{n}\left( \alpha _{n,k}\frac{R_{0}}{R}\right)
+c_{n,k}Y_{n}\left( \alpha _{n,k}\frac{R_{0}}{R}\right) \right] .
\end{eqnarray*}%
Using H\"{o}lder's inequality,
\begin{equation*}
|u_{n,k,l}(r,\varphi )|\leq \sqrt{\frac{\pi }{\Gamma \left( 2\right) }R^{2}-%
\frac{\pi }{\Gamma \left( 2\right) }R_{0}^{2}}\Vert u_{n,k,l}\Vert
_{2},
\end{equation*}%
\noindent where
\begin{eqnarray}
\Vert u_{n,k,l}\Vert _{2}^{2} &=&\frac{\pi (2-\delta
_{n,0})R^{2}}{2\alpha _{n,k}^{2}}\left[ \left( \alpha
_{n,k}^{2}+h^{2}R^{2}-n^{2}\right)
v_{n,k}^{2}(R)\right.   \notag \\
&&\left. -\left( (\alpha _{n,k}^{2}+h^{2}R^{2})\frac{R_{0}^{2}}{R^{2}}%
-n^{2}\right) v_{n,k}^{2}(R_{0})\right] ,  \label{eqb1}
\end{eqnarray}%
\noindent with
\begin{equation}
v_{n,k}(r)=J_{n}(\alpha _{n,k}r/R)+c_{n,k}Y_{n}(\alpha _{n,k}r/R).
\label{eqb2}
\end{equation}%
From equations (\ref{eqb1}) and (\ref{eqb2}),  the uniform upper
bound for the eigenvectors of Dirichlet negative Laplacian operator
on circular annulus is then given by
\begin{equation*}
C(D)=\sqrt{\frac{\pi }{\Gamma \left( 2\right) }R^{2}-\frac{\pi
}{\Gamma \left( 2\right) }R_{0}^{2}}\pi R^{2}(1+M_{1})M_{2},
\end{equation*}%
\noindent where
\begin{eqnarray}
\frac{h^{2}R^{2}}{\alpha _{n,k}^{2}} &\leq &M_{1}  \notag \\
v_{n,k}^{2}(R) &\leq &M_{2},  \notag
\end{eqnarray}%
\noindent since Bessel functions of the first and second kind are
uniformly bounded for $0<R_{0}<r<R.$ Furthermore, Bessel functions
of the first and second kind on a closed bounded interval, as well
as sine and cosine are also H\"older continuous.  Thus, Theorem
\ref{th2}  holds. Summarizing, Theorems \ref{prtqv}, \ref{th2},
\ref{stinc} and \ref{thspp} hold for $\alpha +\gamma>n,$ and $\beta
<1/2.$

\subsubsection*{Elliptical annulus}

 In elliptic coordinates, $x_{1}=a\cosh r\cos \theta,$
$x_{2}=a\sinh r \sin \theta,$ the Laplace operator adopts the form:
\begin{equation*}
\Delta =\frac{1}{a^{2}(\sinh^{2} r+\sin^{2} \theta
)}\left(\frac{\partial ^{2}}{\partial
r^{2}}+\frac{\partial^{2}}{\partial \theta ^{2}}\right),
\end{equation*}
\noindent where $r\geq 0,$ and $0\leq \theta <2\pi,$ are the radial
and angular coordinates, and $a > 0$ is the prescribed distance
between the origin and the foci. An ellipse is a curve of constant
$r = R$ so that its
points $(x_{1}, x_{2})$ satisfy $x_{1}^{2}/ A^{2} + x^{2}/B^{2} = 1,$ where $%
A = a \cosh R,$ and $B = a\sinh R$ are the major and minor
semi-axes, and $R$ denotes the radius of the ellipse. The
eccentricity $e = a/A = 1/\cosh R$ is strictly positive. The
interior of an ellipse is characterized by $0\leq \theta <2\pi$ and
$0\leq r <R.$ An elliptical annulus, the interior between two
ellipses with the same foci, can be characterized in elliptic
coordinates $(r,\theta)$ with
 $R_{0} < r < R$ and $0\leq \theta <2\pi.$

 In elliptic coordinates, the variable separation method,
$u(r,\theta )=g(\theta )f(r),$ leads to the following equations,
after considering that
the two differential equations in $\theta $ and $r$ are equal to a constant $%
\kappa,$
\begin{eqnarray}
& &g^{\prime \prime }(\theta )+(\kappa-2q\cos 2\theta )g(\theta )=0
\label{Mthieueq0} \\
& &f^{\prime \prime }(r)-(\kappa-2q\cosh 2r\theta )f(r)=0.
\label{Mthieueq}
\end{eqnarray}
These equations are respectively known as the Mathieu equation and
the modified Mathieu equation, where $q=\lambda a^{2}/4,$ and the
parameter $\kappa $ is called the characteristic value of Mathieu
functions, whose values lead to a real integer value of the
characteristic exponent $\nu$ of the solution defined according to
Floquet's theorem. The two linearly independent periodic solutions
of equation (\ref{Mthieueq0}) are known as the angular Mathieu
functions, and they are respectively denoted as $\kappa e_{n}(\theta
,q)$ and $se_{n+1}(\theta, q),$ $n=0,1,2,\dots .$ That is, we
consider $\kappa $ such that the characteristic exponent $\nu $
satisfies $\nu(\kappa ,q)\in \mathbb{Z},$ leading to the referred
angular periodic Mathieu functions. There are two linearly
independent oscillatory radial Mathieu functions of the first kind,
solution to equation (\ref{Mthieueq}), respectively denoted as $%
M\kappa_{n}^{(1)}(r,q)$ and $M\kappa_{n}^{(2)}(r,q),$ corresponding
to the same $\kappa $ as $\kappa e_{n}(\theta ,q).$ In addition,
there is two linearly independent oscillatory radial Mathieu
functions of the second kind $Ms_{n+1}^{(1)}(r,q)
$ and $Ms_{n+1}^{(2)}(r,q)$ corresponding to the same $s$ as $%
se_{n+1}(\theta, q)$ (see, for example, Guti\'errez-Vega \emph{et.
al.} \cite{GutierrezVega02}). Thus we have four families
$l=1,2,3,4,$ of eigenfunctions of Laplacian operator in an
elliptical domain:
\begin{eqnarray}
u_{nk1}(r,\theta )&=& \kappa e_{n}(\theta
,q_{nk1})M\kappa_{n}^{(1)}(r,q_{nk1}) \notag
\\
u_{nk1}(r,\theta )&=& \kappa e_{n}(\theta
,q_{nk2})M\kappa_{n}^{(2)}(r,q_{nk2}) \notag
\\
u_{nk1}(r,\theta )&=& se_{n+1}(\theta,
q_{nk3})Ms_{n+1}^{(1)}(r,q_{nk3})
\notag \\
u_{nk1}(r,\theta )&=& se_{n+1}(\theta,
q_{nk4})Ms_{n+1}^{(2)}(r,q_{nk4}). \label{eeigenvectorLaplacian}
\end{eqnarray}

For the elliptical annulus with Dirichlet boundary conditions having radius $%
0<R_{0}<R$, there are eight individual equations defining the
parameter $q$ for each $n=0,1,2,\dots ,$
\begin{eqnarray}
M\kappa_{n}^{(1)}(R,q_{nk1})&=&0;\ M\kappa_{n}^{(2)}(R,q_{nk2})=0\nonumber\\
Ms_{n+1}^{(1)}(R,q_{nk3})&=&0;\ Ms_{n+1}^{(2)}(R,q_{nk4})=0\nonumber\\
M\kappa_{n}^{(1)}(R_{0},q_{nk1})&=&0;\
M\kappa_{n}^{(2)}(R_{0},q_{nk2})=0\nonumber\\
Ms_{n+1}^{(1)}(R_{0},q_{nk3})&=&0;\
Ms_{n+1}^{(2)}(R_{0},q_{nk4})=0.\nonumber
\end{eqnarray}
The fundamental solution (in elliptic coordinates) of
$$\frac{\partial ^{\beta }c}{\partial t^{\beta
}}(t,\mathbf{x})=-(-\Delta_{D} )^{\alpha /2}(I-\Delta_{D})^{\gamma
/2} c(t,\mathbf{x}),$$

 \noindent for $\alpha + \gamma >n,$ with $-\Delta_{D}
=-\frac{1}{a^{2}(\sinh^{2} r+\sin^{2} \theta )}\left(\frac{\partial
^{2}}{\partial r^{2}}+\frac{\partial^{2}}{\partial \theta
^{2}}\right),$ is then given by
\begin{eqnarray}
& & G_{t-s}(r,\theta ,r^{\prime },\theta ^{\prime
})=\sum_{n=0}^{\infty }\sum_{k=1}^{\infty }\sum_{l=1}^{4}E_{\beta
}\left(
-\left([4q_{nkl}/a^{2}]\right)^{\alpha /2}\right.\nonumber\\
& & \hspace*{6cm}\left.\times
\left(1+\left([4q_{nkl}/a^{2}]\right)\right)^{\gamma /2}(t-s)^{\beta
}\right)
\nonumber\\
& & \hspace*{5cm}\times
 u_{n,k,l}(r,\theta )u_{n,k,l}(r^{\prime },\theta ^{\prime
}),\nonumber
\end{eqnarray}

\noindent where $\{u_{n,k,l}\}$ are given in equation
(\ref{eeigenvectorLaplacian}).

 From H\"older inequality
\begin{equation*}
|u_{n,k,l}(r,\theta )|\leq M\left[\pi (a\cosh R)(a\sinh R)-\pi
(a\cosh R_{0})(a\sinh R_{0})\right],
\end{equation*}
\noindent where $M=\sup_{n,k,l}\max_{(r,\theta )}u_{nkl}(r,\theta
),$ since angular Mathieu functions, for $0\leq \theta <2\pi,$ and
radial Mathieu functions of first and second kind, for $0<R_{0} < r
< R,$ are uniformly bounded (see, for example, Guti\'errez-Vega
\emph{et. al.} \cite{GutierrezVega02}).
 The uniform upper bound,  in  Theorems \ref{prtqv}, for the eigenvectors of Dirichlet negative Laplacian operator on $D$ is then given by
\begin{equation*}
C(D)=M\left[ \pi (a\cosh R)(a\sinh R)-\pi (a\cosh R_{0})(a\sinh
R_{0})\right].
\end{equation*}
Finally, since angular Mathieu functions, and radial Mathieu
functions of the first and second kind   on a closed bounded
interval are H\"older continuous functions, Theorem \ref{th2} also
holds. As before, Theorems \ref{prtqv}, \ref{th2}, \ref{stinc} and
\ref{thspp} follow for $\alpha +\gamma>n,$ and $\beta <1/2.$

\section{Fractional polynomials of the Dirichlet negative Laplacian operator on $D$}
\label{FC1}

\setcounter{section}{8}

The results formulated in this paper hold under a more general
scenario. Specifically, in equations (\ref{1.2})--(\ref{1.2b}), we
can replace $(-\Delta_{D} )^{\alpha /2}(I-\Delta_{D} )^{\gamma /2}$
by a fractional elliptic polynomial of the form
\begin{equation}P\left((-\Delta_{D} )^{\alpha /2}(I-\Delta_{D}
)^{\gamma /2}\right)=\sum_{l=0}^{p}c_{l}\left[(-\Delta_{D} )^{\alpha
/2}(I-\Delta_{D} )^{\gamma /2}\right]^{l},\label{eqpol}
 \end{equation}\noindent  of degree $p,$  and with  constant coefficients $c_{l}\geq 0,$ $l=0,\dots,p-1,$ and $c_{p}>0.$  Thus, the following reformulation of equation (\ref{1.2}) is considered
\begin{equation}
\frac{\partial ^{\beta }}{\partial t^{\beta }}c\left(
t,\mathbf{x}\right) +P\left((-\Delta_{D} )^{\alpha /2}(I-\Delta_{D}
)^{\gamma /2}\right)c\left( t,\mathbf{x}\right)=I^{1-\beta
}_{t}\varepsilon \left( t,\mathbf{x}\right),\quad \mathbf{x}\in
D,\label{1.2pp}
\end{equation}

\noindent with the  boundary and initial conditions given in
(\ref{1.2b}), and with \linebreak $P\left((-\Delta_{D} )^{\alpha
/2}(I-\Delta_{D} )^{\gamma /2}\right)$
 being defined  in (\ref{eqpol}). Here, as before, $\varepsilon $ represents Gaussian space-time white noise. The next result provides the extension of the previously established statements, for equations (\ref{1.2})--(\ref{1.2b}), to equation (\ref{1.2pp}).
\begin{theo}
\label{thfg} The following assertions hold:
\begin{itemize}
\item[(i)]For $n<p(\alpha +\gamma ),$ \begin{equation}
\sum_{k=1}^{\infty }E_{\beta }\left( -\lambda _{k}\left(
P\left((-\Delta_{D} )^{\alpha /2}(I-\Delta_{D} )^{\gamma
/2}\right)\right) t^{\beta }\right) <\infty ,  \label{tracep0}
\end{equation}

\noindent  for every $t>0,$ where, for each $k\geq 1$
\begin{equation}
P\left((-\Delta_{D} )^{\alpha /2}(I-\Delta_{D} )^{\gamma
/2}\right)\phi _{k}=\lambda _{k}\left( P\left((-\Delta_{D} )^{\alpha
/2}(I-\Delta_{D} )^{\gamma /2}\right)\right) \phi _{k}.
\label{sebbvv}
\end{equation}

\noindent For $n<p(\alpha +\gamma ),$ the weak-sense solution on
$\overline{H}^{p(\alpha +\gamma )}(D)$ to (\ref{1.2pp}), in the
mean-square sense, with  boundary and initial conditions
(\ref{1.2b}) is then given by

\begin{equation}
c(t,\mathbf{x})=\int_{0}^{t}\int_{D}G^{D}(t,\mathbf{x};s,\mathbf{y})\varepsilon
(s,\mathbf{y})dsd\mathbf{y},  \label{risbb}
\end{equation}
\noindent where the integral is understood in the mean-square sense,
$\varepsilon (s,\mathbf{y})$ is space-time zero-mean Gaussian white
noise as given in equation (\ref{1.2pp}), and, for $t\geq s,$

\begin{eqnarray}
&&G^{D}(t,\mathbf{x};s,\mathbf{y})=\nonumber\\
&& =\sum_{k\geq 1}E_{\beta }\left( -\lambda _{k}\left(
P\left((-\Delta_{D} )^{\alpha /2}(I-\Delta_{D} )^{\gamma
/2}\right)\right) (t-s)^{\beta }\right) \phi _{k}(\mathbf{x})\phi
_{k}(\mathbf{y}),\  \nonumber \\
&&G^{D}(t,\mathbf{x};s,\mathbf{y})=0,\ s>t,  \label{srbb}
\end{eqnarray}
\noindent with $\{\phi _{k}\}_{k\geq 1}$ and $\left\{\lambda
_{k}\left( P\left((-\Delta_{D} )^{\alpha /2}(I-\Delta_{D} )^{\gamma
/2}\right)\right)\right\}_{k\geq 1}$ satisfying (\ref{sebbvv}).
\item[(ii)] For $\beta
<1/2,$ and  $\frac{n}{2} <p(\alpha +\gamma ),$ the following
inequality holds:
\begin{equation}
E[c(t,\mathbf{x})-c(s,\mathbf{x})]^{2}\leq [C(D)]^{2} g(t-s),
\label{eineqti3bb}
\end{equation}
\noindent where $C(D)$ is defined as in Theorem \ref{prtqv}, and
\begin{equation}
g(t-s)=\mathcal{O}\left((t-s)^{\left(1-\frac{\beta n}{p(\alpha
+\gamma)}\right)\wedge (1-\beta )}\right), \quad s\rightarrow t, \
0<s<t.\label{eqth1bb}
\end{equation}

\item[(iii)] For $n<p(\alpha +\gamma),$ assume that the uniform H\"older continuity of the   eigenvectors of the Dirichlet negative Laplacian operator holds, as given in  Theorem \ref{th2},  considering  $\Vert \mathbf{x}-\mathbf{y}\Vert $
sufficiently small,
\begin{equation}
E[c(t,\mathbf{x})-c(t,\mathbf{y})]^{2}\leq
Cg(t)\|\mathbf{x}-\mathbf{y}\|^{2\Upsilon }, \label{si}
\end{equation}%
\noindent where $C$ is given in  Theorem \ref{th2}, and
\begin{eqnarray}
g(t)&=&t^{1-\beta}\sum_{k=1}^{\infty }\frac{\Gamma (1+\beta
)}{\lambda_{k}\left(P\left((-\Delta_{D} )^{\alpha /2}(I-\Delta_{D}
)^{\gamma /2}\right)\right)} ,\quad t>0.\nonumber
\end{eqnarray}
\end{itemize}
\item[(iv)] For $\beta
<1/2,$ and  $n <p(\alpha +\gamma ),$  assume that the  uniform
H\"older continuity of
the eigenvectors of the Dirichlet negative Laplacian operator holds, as $s\rightarrow t,$ $s,t\in (0,T],$ and $\Vert
\mathbf{x}-\mathbf{y}\Vert \rightarrow 0,$
\begin{eqnarray*}
&&E[c(t,\mathbf{x})-c(s,\mathbf{y})]^{2} \\
&\leq &\widetilde{C}(D,T,\beta ,\alpha,\gamma,p,\Upsilon
,n)\|(t,\mathbf{x})-(s,\mathbf{y})\|^{\left(1-\frac{\beta n}{p(\alpha
+\gamma})\right)\wedge (1-\beta )\wedge 2\Upsilon} ,
\end{eqnarray*}
\noindent where $$\widetilde{C}(D,T,\beta ,\alpha,\gamma,p,\Upsilon
,n)=8([C(D)]^{2}\vee
Cg(T))\left(\frac{1}{2}\right)^{1/2\left[\left(1-\frac{\beta
n}{(\alpha +\gamma )p}\right)\wedge (1-\beta )\wedge
2\Upsilon\right]. }.$$
\item[(v)] For $\beta
<1/2,$ and  $n<p(\alpha +\gamma ),$ under  the  uniform H\"older
continuity of the
eigenvectors of the Dirichlet negative Laplacian operator, as $s\rightarrow t,$ $s,t\in (0,T],$ and $\Vert
\mathbf{x}-\mathbf{y}\Vert \rightarrow 0,$
\begin{eqnarray}
&&\sup_{|t-s|<\delta }|c(t,\mathbf{x})-c(s,\mathbf{x})|^{2}  \notag \\
&\leq &\widetilde{Z}\delta ^{\left(1-\frac{\beta n}{(\alpha
+\gamma)p}\right)\wedge (1-\beta )}+\widetilde{H}_{1}\delta
^{\left(1-\frac{\beta n}{(\alpha +\gamma)p}\right)\wedge (1-\beta
)}\left[ \log \left( \frac{1}{\delta }\right) \right] ^{1/2}  \notag
\\
&&\sup_{\Vert \mathbf{x}-\mathbf{y}\Vert <\delta }|c(t,\mathbf{x})-c(t,%
\mathbf{y})|^{2}  \notag \\
&\leq &\widetilde{Y}\delta ^{2\Upsilon}+\widetilde{H}_{2}\delta
^{2\Upsilon}
\left[ \log \left( \frac{1}{\delta }\right) \right] ^{1/2}  \notag \\
&&\sup_{\Vert (t,\mathbf{x})-(s,\mathbf{y})\Vert <\delta }|c(t,\mathbf{x}%
)-c(s,\mathbf{y})|^{2}  \notag \\
&\leq &\widetilde{X}\delta ^{\left(1-\frac{\beta n}{(\alpha
+\gamma)p}\right)\wedge (1-\beta )\wedge 2\Upsilon}  \notag \\
&&+\widetilde{H}_{3}\delta ^{\left(1-\frac{\beta n}{(\alpha
+\gamma)p}\right)\wedge (1-\beta )\wedge 2\Upsilon}\left[ \log
\left( \frac{1}{\delta }\right) \right] ^{1/2},
\notag \\
&&
\end{eqnarray}%
\noindent where $\widetilde{Z},\widetilde{Y}$ and $\widetilde{X}$
are positive random variables, and $\widetilde{H}_{i},$ $i=1,2,3,$
are positive constants that could depend on the geometrical
characteristics of the domain $D$ considered, like the boundary.
\end{theo}

\begin{proof}
\noindent (i) As, in  Corollary \ref{cor1}, we apply equation
(\ref{5}) in Theorem \ref{theorem:1}, considering $f(u)= P(u^{\alpha
/2}(1+u)^{\gamma /2}),$ with $P(u)=\sum_{l=0}^{p}c_{l}u^{l}$ given
in equation (\ref{eqpol}), to obtaining
\begin{equation} \lambda _{k}\left(P\left((-\Delta_{D})^{\alpha
/2}(I-\Delta_{D} )^{\gamma /2}\right)\right)=
\sum_{l=0}^{p}c_{l}\left[\left(\gamma_{k}(-\Delta_{D}
)\right)^{\alpha /2}(1+\gamma_{k}(-\Delta_{D} ))^{\gamma
/2}\right]^{l}, \label{fsbb}
\end{equation}
\noindent where, as before, $\{\gamma_{k}(-\Delta_{D} )\}_{k\geq 1}$
denotes the eigenvalues of the Dirichlet negative Laplacian operator
arranged in decreasing order of their magnitude. Since
\begin{equation}
\gamma_{k}(-\Delta_{D} )\sim 4\pi\frac{\left(\Gamma \left( 1+\frac{n}{2}%
\right)\right)^{2/n}}{|D|^{2/n}}k^{2/n},\quad k\longrightarrow
\infty , \label{tseLapacbb}
\end{equation}
\noindent  (see, for example,  Chen and  Song \cite{ChenSong05}), we
obtain
\begin{equation}
\lim_{k\longrightarrow \infty }\frac{\lambda _{k}\left(
P\left((-\Delta_{D})^{\alpha /2}(I-\Delta_{D} )^{\gamma
/2}\right)\right) }{k^{(\alpha +\gamma )p
/n}}=\widetilde{c}(n,\alpha +\gamma,p )|D|^{-p(\gamma +\alpha )/n},
\label{eqfa2bb}
\end{equation}
\noindent where $\widetilde{c}(n,\alpha +\gamma )$ is a positive
constant depending on $n,$ $\alpha ,$  $\gamma$ and $p.$

Futhermore, equation (\ref{5}) in Theorem \ref{theorem:1} also
implies  the following equality: For $k\geq 1,$
\begin{equation}
P\left((-\Delta_{D})^{\alpha /2}(I-\Delta_{D} )^{\gamma
/2}\right)\phi_{k}=\lambda _{k}\left(P\left((-\Delta_{D})^{\alpha
/2}(I-\Delta_{D} )^{\gamma /2}\right)\right)\phi_{k},
\label{eqeivectorLbb}
\end{equation}
\noindent  for the eigenvector system  $\{\phi_{k}\}_{k\geq 1}$ of
the Dirichlet negative Laplacian operator $(-\Delta_{D} )$ on domain
$D.$

From equation (\ref{eqfa2bb}), there exists $k_{0}$ such that for
$k\geq k_{0},$
\begin{equation}\widetilde{L}_{1}k^{p(\alpha+\gamma )/n}\leq \lambda_{k}\left(P\left((-\Delta_{D} )^{\alpha /2}(I-\Delta_{D} )^{\gamma
/2}\right)\right)\leq \widetilde{L}_{2} k^{p(\alpha+\gamma
)/n},\label{ippbb}
\end{equation} \noindent for
certain positive constants $0<\widetilde{L}_{1}<\widetilde{L}_{2},$
depending on $k_{0},$ and $p(\alpha +\gamma )$ and $n.$  In
particular, for $k\geq k_{0},$
\begin{eqnarray}& & \frac{1}{1+[\Gamma (1+\beta
)]^{-1}\lambda_{k}\left(P\left((-\Delta_{D})^{\alpha
/2}(I-\Delta_{D} )^{\gamma /2}\right)\right)t^{\beta }}\nonumber\\
& & \hspace*{1.5cm}\leq \frac{1}{1+[\Gamma (1+\beta )]^{-1}
\widetilde{L}_{1}k^{p(\alpha+\gamma )/n}t^{\beta }}.
\label{eivaso3bb}
\end{eqnarray}

From  Lemma \ref{th4s}, in a similar way to Proposition \ref{pr1},
for each fixed $t>0,$

\begin{eqnarray}
& & \sum_{k=1}^{\infty }E_{\beta }\left(-t^{\beta
}\lambda_{k}\left(P\left((-\Delta_{D})^{\alpha /2}(I-\Delta_{D}
)^{\gamma
/2}\right)\right)\right)\nonumber\\
& &\leq M(\beta,\alpha,\gamma,p,n)+\frac{t^{-\beta n/p(\alpha
+\gamma)}}{p(\alpha
+\gamma)/n}\int_{0}^{\infty}\frac{u^{\frac{n}{p(\alpha +\gamma  )
}-1}}{1+[\Gamma (1+\beta )]^{-1}u}du<\infty ,\nonumber\\
 \label{pfpbb}
\end{eqnarray}

\noindent since
\begin{equation}M(\beta,\alpha,\gamma,p,n)=\sum_{k=1}^{k_{0}}E_{\beta
}\left(-t^{\beta }\lambda_{k}\left(P\left((-\Delta_{D} )^{\alpha
/2}(I-\Delta_{D} )^{\gamma
/2}\right)\right)\right)<\infty,\label{eqp1bbc}\end{equation}
\noindent and $\int_{0}^{\infty}\frac{u^{\frac{n}{p(\alpha +\gamma
)}-1}}{1+[\Gamma (1+\beta )]^{-1}u}du<\infty,$ for $p(\alpha +\gamma
)
>n.$

Applying triangle and Cauchy-Schwarz inequalities, since $\psi \in \overline{H}^{p(\alpha +\gamma )}(D),$ and, for $p(\alpha
+\gamma)>n$ equation (\ref{pfpbb}) holds, we obtain
\begin{eqnarray}
& & \left|\sum_{k=1}^{\infty} \frac{\partial ^{\beta }}{\partial t^{\beta
}} E_{\beta }\left( - \lambda_{k}\left( P\left((-\Delta_{D}
)^{\alpha /2}(I-\Delta_{D})^{\gamma /2}\right)\right)t^{\beta
}\right) \phi _{k}(\mathbf{x})\psi_{k}\right|\nonumber\\ & & \leq
\sum_{k=1}^{\infty} \lambda_{k}\left( P\left((-\Delta_{D} )^{\alpha
/2}(I-\Delta_{D} )^{\gamma /2}\right)\right) \nonumber\\ &
&\hspace*{1cm}\times E_{\beta }\left( -
\lambda_{k}\left(P\left((-\Delta_{D} )^{\alpha /2}(I-\Delta
_{D})^{\gamma /2}\right)\right)t^{\beta }\right) |\phi
_{k}(\mathbf{x})|\psi_{k}\nonumber\\ & & \leq C(D)\sum_{k=1}^{\infty}
\lambda_{k}\left( P\left((-\Delta _{D})^{\alpha /2}(I-\Delta_{D}
)^{\gamma /2}\right)\right)\nonumber\\ & &\hspace*{1cm}\times
E_{\beta }\left( - \lambda_{k}\left( P\left((-\Delta_{D} )^{\alpha
/2}(I-\Delta_{D} )^{\gamma /2}\right)\right)t^{\beta
}\right)\psi_{k}<\infty,\nonumber\\\label{fsd}
\end{eqnarray}
\noindent where, as before,
$\psi_{k}=\int_{D}\phi_{k}(\mathbf{y})\psi(\mathbf{y})d\mathbf{y}.$

 Applying the regularized fractional derivative in
time (\ref{1.3}),  we then obtain, in a similar way to Proposition
\ref{pr2},

 \begin{eqnarray}
& &\int_{D} \frac{\partial ^{\beta }}{\partial t^{\beta }}c\left(
t,\mathbf{x}\right)\psi(\mathbf{x})d\mathbf{x}
=\int_{D}I_{t}^{1-\beta }\varepsilon
(t,\mathbf{x})\psi(\mathbf{x})d\mathbf{x} \nonumber\\
& &-\int_{D}\left[P\left((-\Delta_{D} )^{\alpha /2}(I-\Delta
_{D})^{\gamma /2}\right)\int_{0}^{t}\int_{D}G^{D}\left(
t,\mathbf{x};s,\mathbf{y}\right)\varepsilon
(s,\mathbf{y})d\mathbf{y}ds\right]
\psi(\mathbf{x})d\mathbf{x},\nonumber\\
\end{eqnarray}
\noindent for  every $\psi \in \overline{H}^{p(\alpha +\gamma
)}(D),$ as we wanted to prove.

\medskip

\noindent (ii) In a similar way to Theorem \ref{prtqv},

\begin{eqnarray}
& &E[c(t,\mathbf{x})-c(s,\mathbf{x})]^{2}\leq
[C(D)]^{2}\int_{0}^{s}\left[
\widetilde{M}(\beta,\alpha,\gamma,p,n,(s-u)^{\beta })\right.
\nonumber\\& & \left. \hspace*{2cm}+\frac{(s-u)^{-\beta n/p(\alpha
+\gamma)}}{p(\alpha
+\gamma)/n}\int_{0}^{\infty}\frac{x^{\frac{n}{p(\alpha +\gamma
)}-1}}{(1+[\Gamma (1+\beta )]^{-1}x)^{2}}dx\right]du\nonumber\\
& & +[C(D)]^{2}\int_{s}^{t}\left[M(\beta,\alpha,\gamma,p,n,
(t-u)^{\beta})\right.\nonumber\\
& & \left.\hspace*{2cm}+\frac{(t-u)^{-\beta n/p(\alpha
+\gamma)}}{p(\alpha +\gamma)/n}
\int_{0}^{\infty}\frac{x^{\frac{n}{p(\alpha +\gamma )}-1}}{1+[\Gamma
(1+\beta )]^{-1}2x}dx\right]du, \nonumber\\\label{ineqfund}
\end{eqnarray}

\noindent where
\begin{eqnarray}
& &\hspace*{-1cm} M(\beta,\alpha,\gamma,p,n,
(s-u)^{\beta})=\sum_{k=1}^{k_{0}}E_{\beta }\left(-2(s-u)^{\beta
}\lambda_{k}\left(P\left((-\Delta_{D} )^{\alpha /2}(I-\Delta_{D}
)^{\gamma /2}\right)\right)\right)
\nonumber\\
& & \hspace*{-1cm}\widetilde{M}(\beta,\alpha,\gamma,p,n,
(s-u)^{\beta })=\sum_{k=1}^{k_{0}}\left[E_{\beta
}\left(-(s-u)^{\beta }\lambda_{k}\left(P\left((-\Delta_{D} )^{\alpha
/2}(I-\Delta_{D} )^{\gamma
/2}\right)\right)\right)\right]^{2}.\nonumber\\\label{eqp1bbcc}\end{eqnarray}

Hence,  \begin{eqnarray}&
&E[c(t,\mathbf{x})-c(s,\mathbf{x})]^{2}\leq [C(D)]^{2}\left[
K_{1}(\beta,\alpha,\gamma,p,n)s^{1-2\beta}\right.\nonumber\\
& & \hspace*{1cm}\left.+K_{2}(\beta,\alpha,\gamma,p,n)s^{1-\beta
n/(p(\alpha +\gamma)
)}+K_{3}(\beta,\alpha,\gamma,p,n)(t-s)^{1-\beta}\right.\nonumber\\
& & \left.\hspace*{1cm}+K_{4}(\beta,\alpha,\gamma,p,n)(t-s)^{1-\beta
n/((\alpha +\gamma )p)}\right]. \label{ineqfundb}
\end{eqnarray}

Thus, we have
$$E[c(t,\mathbf{x})-c(s,\mathbf{x})]^{2}\leq [C(D)]^{2}g(t-s),$$ \noindent with, for $s\rightarrow t,$ $s<t,$ $$g(t-s)=\mathcal{O}\left((t-s)^{1-\left(\beta
n/p(\alpha +\gamma )\right)\wedge (1-\beta )}\right).$$

\medskip

\noindent (iii) Applying   H\"older continuity of the eigenvectors,
 from Lemma
\ref{th4s},  in a  similar way to Theorem \ref{th2},  for every
$t>0,$
\begin{eqnarray}
& &E[c(t,\mathbf{x})-c(t,\mathbf{y})]^{2}
 \leq C\|\mathbf{x}-\mathbf{y}\|^{2\Upsilon
}
\nonumber\\
& & \hspace*{4cm}\times\int_{0}^{t}\sum_{k=1}^{\infty }\frac{\Gamma
(1+\beta )}{\lambda_{k}\left(P\left((-\Delta )^{\alpha
/2}_{D}(I-\Delta )^{\gamma
/2}_{D}\right)\right)\nu^{\beta }}d\nu \nonumber\\
& & =C\|\mathbf{x}-\mathbf{y}\|^{2\Upsilon
}\left[t^{1-\beta}\sum_{k=1}^{\infty }\frac{\Gamma (1+\beta
)}{\lambda_{k}\left(P\left((-\Delta )^{\alpha /2}_{D}(I-\Delta
)^{\gamma /2}_{D}\right)\right)}\right]
\nonumber\\
& &=Cg(t)\|\mathbf{x}-\mathbf{y}\|^{2\Upsilon},
\end{eqnarray}
\noindent \noindent as we wanted to prove.  Here, for each  fixed
$t>0,$
\begin{equation}g(t)=t^{1-\beta}\sum_{k=1}^{\infty }\frac{\Gamma (1+\beta
)}{\lambda_{k}\left(P\left((-\Delta )^{\alpha /2}_{D}(I-\Delta
)^{\gamma /2}_{D}\right)\right)} <\infty, \label{gfunc}
\end{equation}
\noindent for $p(\alpha +\gamma )>n.$

\medskip

\noindent (iv) In a similar way to Theorem \ref{stinc}, since under
the conditions assumed, $0<\left(1-\frac{\beta n}{p(\alpha
+\gamma)}\right)\wedge (1-\beta )\wedge 2\Upsilon <1,$  applying
Jensen's inequality
 we obtain,
\begin{eqnarray}
& &E[c(t,\mathbf{x})-c(s,\mathbf{y})]^{2} \leq 4([C(D)]^{2}\vee
Cg(T))\left(\frac{1}{2}\right)^{1/2\left[\left(1-\frac{\beta
n}{p(\alpha +\gamma )}\right)\wedge (1-\beta )\wedge
2\Upsilon\right]
}\nonumber\\
& &\hspace*{2cm} \times
\left[\left(|t-s|^{2}+\|\mathbf{x}-\mathbf{y}\|^{2}\right)^{1/2\left[\left(1-\frac{\beta
n}{p(\alpha +\gamma)}\right)\wedge
(1-\beta )\wedge 2\Upsilon \right]}\right.\nonumber\\
& &\hspace*{2cm} +\left.\sqrt{\left(|t-s|^{2}
+\|\mathbf{x}-\mathbf{y}\|^{2}\right)^{\left(1-\frac{\beta
n}{p(\alpha +\gamma)}\right)\wedge (1-\beta )\wedge
2\Upsilon}}\right]\nonumber\\
& &=2C(D,T,\beta ,\alpha,\gamma,p, \Upsilon, n
)\|(t,\mathbf{x})-(s,\mathbf{y})\|^{\left(1-\frac{\beta n}{p(\alpha
+\gamma)}\right)\wedge (1-\beta )\wedge 2\Upsilon},
\end{eqnarray}
\noindent where $\widetilde{C}(D,T,\beta ,\alpha,\gamma,p,\Upsilon, n)= 2C(D,T,\beta ,\alpha,\gamma,p,\Upsilon, n),$ and $$C(D,T,\beta
,\alpha,\gamma,p,\Upsilon, n)=4([C(D)]^{2}\vee
Cg(T))\left(\frac{1}{2}\right)^{1/2\left[\left(1-\frac{\beta
n}{p(\alpha +\gamma)}\right)\wedge (1-\beta )\wedge 2\Upsilon\right]
}.$$

\medskip

\noindent (v) The sample-path regularity properties follow
straightforward from (ii)--(iv), by applying Theorem 3.3.3 in
 Adler \cite{Adler81},  p.57.

\hfill $\blacksquare$

\end{proof}

\section{Final comments}
\label{FC}

\setcounter{section}{9}

Under the conditions assumed in Proposition \ref{pr1}, a mean-square
solution to equations (\ref{1.2})-(\ref{1.2b}) is derived in
 Proposition \ref{pr2}, in the weak-sense on the space $\overline{H}^{\alpha +\gamma }(D)$ (respectively on the space $\overline{H}^{p(\alpha +\gamma )}(D),$
   in the general  case considered in Theorem \ref{thfg}). In particular, from the results derived, we can define a $H^{-(\alpha +\gamma )}(D)-$valued
   stochastic process
   $\{\mathcal{C}_{t},\ t\in \mathbb{R}^{+}\},$ on
   the basic probability space $(\Omega,\mathcal{A},P),$
   satisfying  equation (\ref{1.2}) a.s., i.e.,

\begin{eqnarray}
& & \left\langle \mathcal{C}_{t}(\cdot,\omega),\frac{\partial
^{\beta }}{\partial t^{\beta }}+\left( -\Delta_{D} \right) ^{\alpha
/2}\left( I-\Delta_{D} \right) ^{\gamma /2}\psi (\cdot)
\right\rangle_{L^{2}(D)}\nonumber\\ & & =\left\langle I_{t}^{1-\beta
}\varepsilon (\cdot,\omega),\psi
(\cdot)\right\rangle_{L^{2}(D)},\quad \mbox{a.s},\quad \forall
\psi\in \overline{H}^{\alpha +\gamma}(D),\ t\in
\mathbb{R}_{+}.\label{wshv}
\end{eqnarray}
\noindent  (Note that similar assertions hold for the derived
solution to equation (\ref{1.2pp}) in Theorem \ref{thfg}).   In this
derivation, the orthogonality in
$\mathcal{L}^{2}(\Omega,\mathcal{A},P)$ of the random components of
$\varepsilon$
   is applied, i.e., we have applied that $E[\varepsilon(t,\mathbf{x})]=0,$ and $E[\varepsilon(t,\mathbf{x})\varepsilon (s,\mathbf{y})]=\delta (t-s)\delta (\mathbf{x}-\mathbf{y}),$ for  $t,s\in \mathbb{R}_{+},$
 and $\mathbf{x},\mathbf{y}\in D\subset \mathbb{R}^{n}.$ It is well-known that this property holds for any white noise measure on $L^{2}(\mathbb{R}_{+}\times D),$ beyond the Gaussian case.
 Furthermore, the fractional integration in the definition of the driven process $I^{1-\beta }_{t}\varepsilon $ in equation (\ref{1.2})  is understood in
 the mean-square sense on a suitable space of test functions, as given in Proposition \ref{pr2}. Thus, we have only considered the properties of the second-order moments
 of the distribution of the driven process in equation $(\ref{1.2}).$  Hence, Proposition \ref{pr2} also  holds when the Gaussian space-time white noise on $L^{2}(\mathbb{R}_{+}\times D)$
 is replaced by an arbitrary white noise random measure $\widetilde{\varepsilon}$ on $L^{2}(\mathbb{R}_{+}\times D).$ In particular, L\'evy noise can be considered. In that case, Theorems \ref{prtqv},
  \ref{th2} and \ref{stinc} respectively
 provide the H\"older continuity, in the mean-square sense (i.e., the continuity of the second-order moments), in time, space, and space and time  of the
 weak-sense solution, defined by   integration with respect  to L\'evy noise measure $d\eta,$ as  \begin{equation}
c(t,\mathbf{x})=\int_{0}^{t}\int_{D}G^{D}(t,\mathbf{x};s,\mathbf{y})d\eta
(s,\mathbf{y}),  \label{ris2}
\end{equation}
\noindent with, as before, $G^{D}$ being defined in  (\ref{sr}), for
the case of Proposition  \ref{pr2}, and in (\ref{srbb}), for the
case of Theorem \ref{thfg}. In  both cases, we can interpret the
integral (\ref{ris2}) as a multiparameter It$\widehat{\mbox{o}}$
integral with respect to  $n+1$-parameter L\'evy process $\eta ,$
since, for $\alpha +\gamma >n,$ and for each $t\in \mathbb{R}_{+},$
$G^{D}_{t}$ defines a trace operator $\mathcal{G}_{t}$ on
$L^{2}(D),$ as proved in Proposition \ref{pr1} (see, for example,
L{\o}kka, {\O}ksendal and Proske \cite{Proske}, for an alternative
interpretation and derivation of solutions in that L\'evy noise
case, in terms of functions with values in the Kondratiev space  of
stochastic distributions). Summarizing, the derived results provide
the characterization of the second-order regularity properties of
the weak-sense solution to equations (\ref{1.2})--(\ref{1.2b})
(respectively, to equation (\ref{1.2pp}) in Theorem \ref{thfg}). For
the non-Gaussian case, further research should be developed in order
to obtain the distributional characteristics of (\ref{ris2}), beyond
the second-order moments. This subject will be considered in a
subsequent paper. Note also that  Theorem \ref{thspp}, on the
characterization of the sample-path regularity properties of the
weak-sense solution to equations (\ref{1.2})-(\ref{1.2b}), in the
mean-square sense (respectively,  (v) of Theorem \ref{thfg}) only
holds for the Gaussian case.

The authors have  recently got new results on  fractional-in-time and
   multifractional-in-space stochastic partial differential equations. Such results 
   also appear in the Journal \textbf{Fractional Calculus \& Applied Analysis, Vol. 19, 
   pp. 1434--1459 , DOI: 10.1515/fca-2016-0074,  and  is available online at http://www.degruyter.com/view/j/fca}.

\section*{Acknowledgements}
This work has been supported in part by projects MTM2012-32674  and
MTM2015--71839--P (co-funded with Feder funds), of the DGI, MINECO,
Spain. N. Leonenko was supported in particular by Cardiff Incoming
Visiting Fellowship Scheme and International Collaboration Seedcorn
Fund and Australian Research Council's Discovery Projects funding
scheme (project number DP160101366).

 %%%%%%%%%%%%%%%%%%%%%%%%%%%%%%%%

\medskip

{\small
\noindent Vo V. Anh\\
\noindent Queensland University of Technology\\
\noindent e-mail: v.anh@qut.edu.au\\
\noindent School of Mathematical Sciences\\
\noindent GPO Box 2434, Brisbane, QLD 4001, Australia\\

\noindent
Nikolai N. Leonenko\\
\noindent
Cardiff University\\
\noindent
e-mail: leonenkon@cardiff.ac.uk\\
\noindent
Mathematics Institute\\
 Senghennydd Road,  Cardiff, CF24 4AG, U.K.\\

\noindent Mar\'{\i}a D. Ruiz-Medina\\
\noindent University of Granada\\
\noindent e-mail: mruiz@ugr.es\\
\noindent  Faculty of Sciences\\
\noindent  C/ Fuente Nueva s/n, 18071 Granada, Spain\\
}\end{document}